\documentclass{amsart}
\usepackage{graphicx}
\usepackage{amssymb}
\usepackage{epstopdf}
\usepackage{nicefrac}
\usepackage{enumerate}
\usepackage{hyperref}
\usepackage{float}
\usepackage{color}
\usepackage{pst-all}
\usepackage{tabularx}

\usepackage{graphicx}

\input xy
\xyoption {all}

\newtheorem{thm}{THEOREM}[section]
\newtheorem{conj}[thm]{CONJECTURE}
\newtheorem{cor}[thm]{COROLLARY}

\newtheorem{defn}[thm]{DEFINITION}
\newtheorem{ex}[thm]{EXAMPLE}

\newtheorem{lemma}[thm]{LEMMA}
\newtheorem{problem}[thm]{PROBLEM}
\newtheorem{prop}[thm]{PROPOSITION}

\newtheorem{remark}[thm]{REMARK}

\newcommand{\ds}{\displaystyle}

\newcommand{\cD}{{\mathcal D}}
\newcommand{\cE}{{\mathcal E}}
\newcommand{\cG}{{\mathcal G}}
\newcommand{\cH}{{\mathcal H}}

\newcommand{\cO}{{\mathcal O}}
\newcommand{\cP}{{\mathcal P}}

\newcommand{\e}{{\varepsilon}} 
\newcommand{\fA}{{\mathfrak{A}}}
\newcommand{\fD}{{\mathfrak{D}}}

\newcommand{\fG}{{\mathfrak{G}}}
\newcommand{\fI}{{\mathfrak{I}}}

\newcommand{\fM}{\mathfrak{M}}
\newcommand{\Fix}{{\rm Fix}} 
\newcommand{\fN}{{\mathfrak{N}}}
\newcommand{\fQ}{{\mathfrak{Q}}}
\newcommand{\fR}{{\mathfrak{R}}}
\newcommand{\fS}{{\mathfrak{S}}}

\newcommand{\fX}{{\mathfrak{X}}}

\newcommand{\G}{\Gamma}
\newcommand{\F}{{\mathcal F}}

\newcommand{\Aut}{{\rm Aut}} 
\newcommand{\Homeo}{{\rm Homeo}}

\newcommand{\Iso}{{\rm Iso}} 

\newcommand{\mQ}{{\mathbb Q}}
\newcommand{\mR}{{\mathbb R}}
\newcommand{\mS}{{\mathbb S}}
\newcommand{\mT}{{\mathbb T}}
\newcommand{\mZ}{{\mathbb Z}}
\newcommand{\oH}{\overline{H}}
\newcommand{\oG}{\overline{\Gamma}}
\newcommand{\ophi}{\overline{\phi}}
\newcommand{\oPhi}{\overline{\Phi}}

\newcommand{\vp}{{\varphi}}

\newcommand{\whg}{\widehat{g}}

\newcommand{\whh}{\widehat{h}}
\newcommand{\whk}{\widehat{k}}
\newcommand{\whp}{\widehat{p}}
\newcommand{\whq}{\widehat{q}}
\newcommand{\whx}{\widehat{x}}
\newcommand{\why}{\widehat{y}}
\newcommand{\whz}{\widehat{z}}

\newcommand{\whGamma}{\widehat{\Gamma}}

\newcommand{\whG}{\widehat{\Gamma}}

\newcommand{\whM}{\widehat{M}}

\newcommand{\whmZ}{\widehat{\mZ}}
\newcommand{\whphi}{\widehat{\phi}}

\newcommand{\whmT}{{\widehat{\mT}}}

\newcommand{\ok}{\overline{k}}

\newcommand{\obeta}{\overline{\beta}}
\newcommand{\ogamma}{\overline{\gamma}}
\newcommand{\omu}{\overline{\mu}}

\newcommand{\FP}{{\mathcal F}_{\cP}} 

\newcommand{\MCG}{{\rm Mod}}
\newcommand{\Comm}{{\rm Comm}}

\newcommand{\whpsi}{{\widehat{\psi}}}
\newcommand{\whvp}{{\widehat{\varphi}}}

\newcommand{\wtgamma}{\widetilde{\gamma}}

\newcommand{\wtM}{\widetilde{M}}

\newcommand{\wth}{\widetilde{h}}
\newcommand{\wtx}{\widetilde{x}}

\newcommand{\wtH}{\widetilde{H}}

\newcommand\mor{\mathrel{\stackrel{\makebox[0pt]{\mbox{\normalfont\tiny a}}}{\sim}}}
\newcommand\comm{\mathrel{\stackrel{\makebox[0pt]{\mbox{\normalfont\tiny c}}}{\sim}}}

\begin{document}

\title{Mapping class groups  of solenoidal manifolds} \thanks{This manuscript was typeset and   spell-checked using LaTeX, and references found using the arXiv and MathSciNet. No other computer resources or AI programs  were used in its preparation. \\
The authors are grateful  to Daniel Studenmund for   helpful conversations.\\
OL is partially supported by the Dutch Research Council (NWO) under the grant OCENW.M.24.200 (grant ID https://doi.org/10.61686/WMQWN26662)}  
 
\author{Steven Hurder}
\address{Steven Hurder, Department of Mathematics, University of Illinois at Chicago, 322 SEO (m/c 249), 851 S. Morgan Street, Chicago, IL 60607-7045}
\email{hurder@uic.edu}

\author{Olga Lukina}
\address{Olga Lukina, Mathematical Institute, 
  Leiden University,
P.O. Box 9512,
2300 RA Leiden,
The Netherlands}
\email{o.lukina@math.leidenuniv.nl}

\thanks{Version date: October 10, 2026}

\thanks{2020 {\it Mathematics Subject Classification}.   
  Primary: 20E18, 20F28 , 37B05  ;  Secondary: 20F65, 57K20, 57S25}
 
\thanks{Keywords: odometer actions, mapping class group, solenoidal manifold, foliated spaces,   virtual automorphism, abstract commensurator.}

\begin{abstract} The goal of this work is to calculate the   mapping class groups of  solenoidal manifolds. We introduce  the dynamical commensurator group associated to an odometer action, and develop   methods to calculate the mapping class group using the dynamical commensurator.     Our results generalize the previous results of   Odden for the mapping class group of the universal solenoid over a Riemann surface, and  of   Bering and Studenmund for the pointed mapping class group of the universal solenoid over a compact aspherical manifold.  We include a selection of examples.
\end{abstract}

\maketitle

\section{Introduction}\label{sec-intro}

 Let $\fM$ be a  compact connected metric space, and $\Homeo(\fM)$ denote its group of homeomorphisms, equipped with the compact open topology. The group $\Homeo(\fM)$ is huge (unless $\fM$ is a singleton) and   understanding it is a challenging  problem.
 Fortunately, there is a reduction of this problem to the one which is not just manageable, but reveals many interesting phenomena.  
 
 Let $\Homeo_0(\fM)$  denote the   connected component of the identity in $\Homeo(\fM)$, which  is a closed normal subgroup. 
 The \emph{mapping class group} of $\fM$ is   the quotient 
 \begin{equation}
\MCG(\fM) = \Homeo(\fM)/\Homeo_0(\fM)   ~ .
\end{equation}
When $\fM = M$ is a closed manifold,   $\MCG(M)$ is   well-understood in special cases.  For example,  the  mapping class groups   $\MCG(\mS^2)$ and $\MCG(\mT^2)$ are completely known.  For a closed surface $\Sigma_g$ of genus $g \geq  2$, the structure of $\MCG(\Sigma_g)$ is far more complex, and has been extensively investigated, as described for example in the treatise  
\cite{FarbMargalit2012}.

The mapping class groups  for  shift spaces derived from dynamical systems   were studied by   Boyle and Chuysurichay \cite{BoyleChuysurichay2018}, and   also  later by Schmiedling and Yang \cite{SchmiedlingYang2021}.  They studied $\MCG(\Sigma_T \fX)$, where   $\Sigma_T \fX$ is the suspension   of a  non-trivial irreducible shift $T \colon \fX \to \fX$ of finite type,  and   obtained many remarkable properties of these groups.

In this work, we study $\MCG(\fM_{\cP})$   where $\fM_{\cP}$ is a \emph{solenoidal manifold}, which is homeomorphic to the  the suspension of an equicontinuous minimal action of a group $\G$ on a Cantor set $\fX$. Recall that a  solenoidal manifold   is  defined as the inverse limit of   a tower of proper coverings,
$$
   \fM_{\cP} \equiv  \varprojlim \,\{\, p_{\ell} \colon  M_{\ell} \to M_{\ell -1} \mid  \ell \geq 1\} ~  .
$$
    Here, each $M_\ell$ is a compact connected manifold of dimension $n$, and $p_{\ell}$ is a  covering map of finite degree greater than one. Note that the maps $p_{\ell}$ are not assumed to be regular (or normal) coverings. The collection $\cP = \{p_{\ell} \colon  M_{\ell} \to M_{\ell -1} \mid  \ell \geq 1\}$ is called a \emph{presentation} of the solenoid.  The   manifold $M_0$ is called the \emph{base}, and the   solenoid $ \fM_{\cP}$ is said to have dimension $n$. McCord showed in \cite{McCord1965} that such  $\fM_{\cP}$ is a foliated space   in the sense of \cite{MS2006}, with foliation $\F_{\cP}$ whose leaves are its path connected components  and have dimension $n$.

For a closed surface $\Sigma_g$ of genus $g \geq 2$, Sullivan considered the space   $\widehat{\Sigma_g}$ formed from the inverse limit of all finite coverings of $\Sigma_g$. This  is called the \emph{universal hyperbolic solenoid}. The space of hyperbolic metrics on the leaves of the   foliation by surfaces of $\widehat{\Sigma_g}$ gives rise to  the universal Teichm\"{u}ller space, which was introduced and studied in the   works of  Biswas, Nag and Sullivan   \cite{BNS1996,BiswasNag1998,BiswasNag2000,Sullivan1993}.  Odden   calculated the group $\MCG(\widehat{\Sigma_g})$ completely in \cite{Odden2005}.

A key idea introduced in the work of Odden is to consider the subgroup $\MCG(\fM_{\cP} , \whx) \subset  \MCG(\fM_{\cP})$
of isotopy classes of homeomorphisms that fix a basepoint $\whx \in \fM_{\cP}$. The leaf $L_{\whx} \subset \fM_{\cP}$ of the foliation of $\fM_{\cP}$ which contains $\whx$ is called the \emph{baseleaf}, and is   invariant under   $h \in \Homeo(\fM_{\cP}, \whx)$, so one can consider the restricted homeomorphism $h \colon L_{\whx} \to L_{\whx}$. If one considers the foliation $\F_{\cP}$ as a dynamical system, then $h \colon  L_{\whx} \to L_{\whx}$ can be considered as a ``non-linear time change'' for an orbit of the dynamical system. Correspondingly,   the map $h \in \MCG(\fM_{\cP} , \whx)$ is a ``global time change''. The strategy for analyzing $\MCG(\fM_{\cP})$ is then to first understand the group $\MCG(\fM_{\cP} , \whx)$ generated by the ``time changes'', and then use these results to calculate $\MCG(\fM_{\cP})$.

Belk and Forrest \cite{BelkForrest2014} and  Bering and Studenmund \cite{BeringStudenmund2024}   considered the group of pointed homotopy self-equivalences  $ \cE(\whM , \whx)$,  where $\whM$ is   the universal solenoidal manifold   over a compact connected $M$.   That is,  $\whM$ is the inverse limit of all finite coverings of $M$. For a manifold $M$ which has contractible universal cover, they relate the group $ \cE(\whM , \whx)$    with the abstract commensurator group $\Comm(\G)$, for $\G = \pi_1(M,x)$. Their work is one of the  motivations for this work, in which we address:
 \begin{problem}\label{prob-basic0} Let $M_0$ be a closed manifold, and let $\fM_{\cP}$ be a solenoidal manifold over $M_0$ with a presentation $\cP$ and a choice of basepoint $\whx \in \fM_{\cP}$. 
Give an algebraic model for $\MCG(\fM_{\cP}, \whx)$   which is    computable in terms of the   topology of $M_0$ and the   properties of the tower of coverings  $\cP$, and use this to calculate $\MCG(\fM_{\cP})$.
  \end{problem}

Recall that a topological space $X$ is \emph{homogeneous} if for every pair of points $x,y \in X$ there exists $h \in \Homeo(X)$ such that $h(x) = y$. 
  Fokkink and Oversteegen showed in \cite{FO2002} that $\fM_{\cP}$ is a homogeneous space  exactly when   the fiber $\fX_0$ of the fibration $\whq_0 \colon \fM_{\cP} \to M_0$  is a profinite group $\fG$, and the monodromy of the fibration $\whq_0$ is via a group representation into $\fG$. In particular, the universal solenoid is always a homogeneous space, see Section \ref{subsec-universal}. Theorem~\ref{thm-homogeneous} and Corollary~\ref{cor-homogeneous} give the      solutions to Problem~\ref{prob-basic0} when   $\fM_{\cP}$ is a homogeneous space.

 When   $\fM_{\cP}$ is  not homogeneous, the study of its  mapping class group  $\MCG(\fM_{\cP})$  requires making 
   assumptions on the dynamical properties of the monodromy of the fibration $\fM_{\cP} \to M_0$.
     In addition to the ideas introduced in the   works \cite{Odden2005,BelkForrest2014,BeringStudenmund2024}, the analysis of this  general case introduces  the \emph{dynamical commensurator group} for a generalized  odometer action $\Phi \colon \G \times \fX \to \fX$, and   develops  its relation to the abstract  commensurator group $\Comm(\G)$  using ideas from the   dynamics of   Cantor actions  \cite{CHL2019a,DHL2016a,DHL2016c}.   Various partial solutions to Problem~\ref{prob-basic0} for non-homogeneous solenoidal manifolds  are developed in this work, in particular as given   in Theorem~\ref{thm-nonhomogeneous} and Corollary~\ref{cor-nonhomogeneous}.

\section{Statement of Results} \label{sec-results}

   Let $M_0$ be a compact connected manifold without boundary, and set  $\G = \pi_1(M_0,x_0)$ for a choice of basepoint $x_0 \in M_0$. 
 A  \emph{group chain} in $\G$ is a  strictly descending chain of subgroups of finite index,  $\cG = \{\G = \G_0 \supset \G_1 \supset \G_2 \supset \cdots\}$.
 The group chain $\cG$ and basepoint $x_0$ determine a \emph{tower} of finite index coverings of $M_0$ by closed manifolds
\begin{equation}\label{eq-presentation}
\cP_{\cG}  \equiv ~  \left\{ M_0   \stackrel{p_1}{\longleftarrow}    M_1    \stackrel{p_2}{\longleftarrow}  M_2   \stackrel{p_3}{\longleftarrow}      \cdots  \right\},
\end{equation}
where $M_0$ is called the base of the tower. The inverse limit space 
\begin{equation}\label{eq-invlimit}
\fM_{\cP} = \lim_{\longleftarrow} ~ \{p_{\ell+1} \colon M_{\ell +1} \to M_{\ell}\}  \subset \prod_{\ell \geq 0} ~ M_{\ell}
\end{equation}
is called  a \emph{weak solenoid} by McCord \cite{McCord1965},  or  a \emph{solenoidal manifold} by Sullivan \cite{Sullivan2014,Verjovsky2014}.

The tower   of coverings $\cP_{\cG}$ is said to be a \emph{presentation} for $\fM_{\cP}$, which is determined by the group chain $\cG$. 
The   projection  $\whq_{0} \colon \fM_{\cP} \to M_{0}$  is a fibration, whose fiber   $\fX_0  = \whq_{0}^{-1}(x_{0})$    is a Cantor space. That is, $\fX_0$ is a compact totally disconnected perfect metrizable topological space.

As stated in the introduction,  our approach to Problem \ref{prob-basic0} follows the works of  Odden \cite{Odden2005}, Belk and Forrest \cite{BelkForrest2014},  and Bering and Studenmund \cite{BeringStudenmund2024}. The first step is to  calculate 
  the \emph{pointed mapping class group} of $\fM_{\cP}$. 
Let $\Homeo(\fM_{\cP}; \whx)$ denote the subgroup of  $\Homeo(\fM_{\cP})$ which   fixes a basepoint $\whx \in \fM_{\cP}$, and set
\begin{equation}
\MCG(\fM_{\cP} ; \whx) = \Homeo(\fM_{\cP}; \whx)/\Homeo_0(\fM_{\cP}; \whx) ~ .
\end{equation}
   The key idea is to relate  $\MCG(\fM_{\cP} ; \whx)$ with the abstract commensurator group of $\G = \pi(M_0,x_0)$.

    A \emph{commensurator} for a countable group $\G$ is a pair of finite-index subgroups $H, K \subset \G$ and an isomorphism $\phi \colon H \to K$.
  Two commensurators $\phi_1 \colon H_1 \to K_1$ and $\phi_2 \colon H_2 \to K_2$ are equivalent, $\phi_1 \sim \phi_2$, if there exists a finite index subgroup $H_3 \subset H_1 \cap H_2$ such that $\phi_1 | H_3 = \phi_2 | H_3$.
  
\begin{defn}\label{def-comm}
 The    set of all
commensurators for $\G$, modulo  the equivalence  relation  $\sim$,
is called the \emph{abstract commensurator group} $\Comm(\G)$. Let $[\phi]$ denote the equivalence class of a commensurator $\phi$.
 \end{defn}

  Let $\G' \subset \G$ be a finite index subgroup. Given    a commensurator $[\phi] \in \Comm(\G)$ represented by  $\phi \colon H \to K$, observe that $\phi \colon H \cap \G' \to K \cap \G'$ defines a commensurator in $\Comm(\G')$. 
The converse   holds, hence  $\Comm(\G') \cong \Comm(\G)$.
Thus,   $\Comm(\G)$ can be intuitively viewed as the group of ``germs''
of isomorphisms between finite-index subgroups of $\G$.   The surveys \cite{Studenmund2015,Studenmund2024}  discuss    properties and calculations of abstract commensurators  for many classes of groups.

The fundamental idea behind the calculation of $\MCG(\fM_{\cP} ; \whx)$  is a result from pro-homotopy theory.
 Mioduszewski showed in  \cite{Mioduszewski1963} that continuous maps between inverse limits of polyhedra have a special form, which up to homotopy, gives  a ``zig-zag correspondence'' between the towers of maps.  Rogers and Tollefson gave a sharper version of this result in \cite{RT1972b} for inverse limits of covering spaces, and in particular for a homeomorphism  $h \in \Homeo(\fM_{\cP}; \whx)$. This result implies that $h$ induces a sequence of   maps between covering spaces of the base $M$, and hence   induces maps on their fundamental groups. This naturally leads to the appearance of commensurators in the study of $\MCG(\fM_{\cP} ; \whx)$.

  \begin{thm}\label{thm-chi}
 Let $\fM_{\cP}$ be a solenoidal manifold with base $M_0$ defined by a group chain $\cG$ in $\G = \pi_1(M_0, x_0)$. Then there is a well-defined characteristic map
\begin{equation}\label{eq-chi}
\chi \colon \MCG(\fM_{\cP} ; \whx)  \to \Comm(\G) \ .
\end{equation}
 \end{thm}
 
The construction of the map $\chi$ is given in Section~\ref{sec-char}.  The commensurator $\chi([h]) \in \Comm(\G)$ is the ``virtual time change'' associated to the homeomorphism $h$.

The map $\chi$  need not be   surjective. A  commensurator $[\phi] \in \Comm(\G)$  is in the image of $\chi$ only if it is represented by a map which preserves the group chain defining the tower of  coverings   $\cP$, up to ``zig-zag'' equivalence.

The map $\chi$     need not be injective, as there may be non-trivial $h \in \MCG(\fM_{\cP} ; \whx)$ which acts  as the identity on a finite index subgroup of $\G$, and hence $\chi(h)$ is the identity.

The first step towards solving Problem~\ref{prob-basic0} is then to solve the following:

 \begin{problem}\label{prob-chi} 
 Determine the image of the characteristic map $\chi \colon \MCG(\fM_{\cP} ; \whx)  \to \Comm(\G)$, and give conditions on the manifold $M_0$  and group  chain $\cG$  which imply that  the map $\chi$ is injective.
\end{problem}
 
Our approach is based on a result from the classification theory of solenoidal manifolds    in \cite{CHL2019a}.
 A solenoidal manifold $\fM_{\cP}$ has an associated Cantor action $\Phi \colon \G \times \fX_0 \to \fX_0$ of $\G$ on the fiber $\fX_0$, where $\Phi$  is defined by the monodromy action of  the fibration $\whq_0 \colon \fM_{\cP} \to M_0$.   This  monodromy action defines an \emph{odometer} in the sense of  \cite{CortezPetite2008,CortezMedynets2016}; see   Section~\ref{sec-odometers}. Every odometer action of a finitely-presented group can be realized as the monodromy of a solenoidal manifold; see Section~\ref{sec-classify} or   \cite{CandelConlon2000,CHL2019a,MS2006}.  The study of odometer actions is a central aspect of  the classification theory of solenoidal manifolds.
 The notion of a \emph{dynamical commensurator} for an odometer   $(\fX, \G, \Phi)$ is a natural consequence of this study.

 First, we recall a basic definition.  An odometer action $(\fX,\G,\Phi)$ admits arbitrarily small clopen sets $U \subset \fX$ such that the collection $\{\Phi(g)(U) : g \in G\}$ forms a partition of $\fX$. We call a set $U$ with this property an \emph{adapted set}.
 
 \begin{defn}\label{def-dyncomm-point}
 Let  $x \in \fX$, then a homeomorphism $h_U \colon U \to V$  is said to be a \emph{dynamical commensurator at} $x$, if $U, V \subset \fX$ are adapted subsets with $x \in U \cap V$,  $h_U(x) = x$,  and $h_U$ induces an isomorphism $\Theta_U \colon \cH_U \to \cH_V$ of their restricted holonomy actions.
 \end{defn}
 
  Two commensurators at $x$, $h_U \colon U \to V$ and $h_{U'} \colon U' \to V'$, are equivalent if there exists an adapted set $U''$ with $x \in U'' \subset U \cap U'$ such that $h_U|U'' = h'_{U'}|U''$. We   write $(h_U,U,V) \comm (h'_{U'},U',V')$.

\begin{defn} \label{def-dyncomm}
Let $(\fX, \G, \Phi)$ be an odometer. For  $x \in \fX$, the \emph{dynamical commensurator group} of $(\fX, \G, \Phi)$ at $x$ is the set of germs of dynamical commensurators,
\begin{equation}\label{eq-dyncomm}
\Comm(\fX, \G, \Phi, x) = \{ h \colon U \to V \mid x \in U \cap V\} \slash  \comm  \  .
\end{equation}
 \end{defn}
A dynamical commensurator is the analog of  ``Morita equivalence'' in the category of foliated spaces, and correspondingly,   $\Comm(\fX, \G, \Phi, x)$ is the group of pointed Morita equivalences. The method for calculating  the groups $\Comm(\fX, \G, \Phi, x)$ in examples is discussed in Section~\ref{subsec-dyncomm}.

 Given  a solenoidal manifold $\fM_{\cP}$, let     $(\fX_0, \G, \Phi)$ denote the associated    monodromy fiber action. 
Given a choice of basepoint $\whx \in \fM_{\cP}$   let $x  \in \fX_0$ denote the corresponding point over $x_0 = \whq_0(\whx) \in M_0$.

The following   result  follows from     \cite{CHL2019a}, and   is discussed in more detail in Section \ref{sec-odometers}.

\begin{thm}\label{thm-charmap}
Suppose that $\fM_{\cP}$ and $\fM'_{\cP'}$ are   solenoidal manifolds.
  If $\fM_{\cP}$ and $\fM'_{\cP'}$ are homeomorphic,  then their monodromy odometers  $(\fX_0, \G, \Phi)$ and  $(\fX_0', \G', \Phi')$ are  return equivalent. Moreover, for $x \in \fX_0$  corresponding to $\whx \in \fM_{\cP}$, there is a well-defined characteristic map
  \begin{equation}
\sigma_{\cP,  \whx} \colon \MCG(\fM_{\cP} ; \whx) \to \Comm(\fX_0, \G, \Phi, x) \ .
\end{equation}
\end{thm}

   Theorems~\ref{thm-chi} and \ref{thm-charmap} yield  maps $\chi$ and $\sigma_{\cP,  \whx}$    in  the  diagram
  \begin{align} \label{eq-diagram}
 \xymatrix{ 
   &  \Comm(\G)        \\
  \MCG(\fM_{\cP} ; \whx)  \ar[ru]^{\chi}  \ar[rd]^{\sigma_{\cP,  \whx}}&      \\
  & \Comm(\fX_0, \G, \Phi, x)     \ar@{-->}[uu]_{\Theta_{\cP} \ ?}
} 
\end{align}
\begin{defn}\label{def-regular}
An odometer action $(\fX_0, \G, \Phi)$ is \emph{regular} if there exists  
\begin{equation}\label{eq-regular}
\Theta_{\cP} \colon \Comm(\fX_0, \G, \Phi, x) \to \Comm(\G)
\end{equation}
 such that the diagram \eqref{eq-diagram} commutes, and $\Theta_{\cP}$ is an injection. 
\end{defn}
Our approach to Problem~\ref{prob-chi} is then based on solving the following two problems.
 \begin{problem}\label{prob-lifting}
 Give conditions on   $\cP_{\cG}$  which suffice to imply that $(\fX_0, \G, \Phi)$ is regular.
  \end{problem}
Note that   Problem~\ref{prob-lifting} only involves the  local automorphism  properties of the dynamical system $(\fX_0, \G, \Phi)$. 
The second problem involves topological aspects of $M_0$ and its coverings. 
 \begin{problem}\label{prob-surjective}
 Give conditions on $M_0$ and the tower of coverings $\cP_{\cG}$ which suffice to imply that $\sigma_{\cP,  \whx}$ is surjective, or better, is  an isomorphism.
 \end{problem}

Solutions to Problems~\ref{prob-lifting} and \ref{prob-surjective}  reduce the calculation of   $\MCG(\fM_{\cP} ; \whx)$ to calculating the   group  $\Comm(\fX_0, \G, \Phi, x)$ of germs of pointed isomorphisms of the odometer action $(\fX_0, \G, \Phi)$  associated to the fibration $\fM_{\cP} \to M_0$.  There are several  approaches to solving  these problems.

The works   \cite{BelkForrest2014,BeringStudenmund2024,Odden2005} consider the case of \emph{universal solenoids}, which we briefly describe.    The collection of all finite-index subgroups of $\G$ forms a partially ordered set $\cG_u(\G)$, with maps defined by inclusions. In this situation, there is a group chain $\cG_{\nu}(\G) \subset \cG_u(\G)$ of characteristic normal subgroups of $\G$ which is cofinal in the set of all finite index subgroups. Then  the \emph{universal solenoid} over $M_0$ is defined to be $\widehat{M_0} = \fM_{\cP}$ where $\cP_{\cG_\nu}$ is the tower of coverings of $M_0$ defined by $\cG_{\nu}(\G)$. The fiber of     $\whq_0 \colon \widehat{M_0}  \to M_0$ is, up to isomorphism, the profinite completion $\whG$ of $\G$.

For a closed surface $\Sigma_g$ of genus $g \geq 2$, the space $\widehat{\Sigma_g}$ is called the \emph{universal hyperbolic solenoid}. This space was introduced in the   works of  Biswas, Nag and Sullivan   \cite{BNS1996,BiswasNag1998,BiswasNag2000,Sullivan1993}  in their study of the universal Teichm\"{u}ller space  of $\widehat{\Sigma_g}$. Odden proved in his thesis the seminal result:
\begin{thm}[Odden, \cite{Odden2005}] \label{thm-odden}
The map $\chi \colon \MCG(\widehat{\Sigma_g} ; \whx)  \to \Comm(\G)$ is an isomorphism.
\end{thm}

More generally,  suppose that $M_0$ is a closed aspherical manifold, that is, the universal cover of $M_0$ is contractible.    Let $\widehat{M_0}$ denote  the universal solenoid over $M_0$.  The work of   Bering and Studenmund \cite{BeringStudenmund2024}   considered the  group  $\cE(\widehat{M_0}  ; \whx)$  of pointed homotopy self-equivalences of $\widehat{M_0}$.
 
\begin{thm}[Bering \& Studenmund,  \cite{BeringStudenmund2024}] \label{thm-BS2024} Assume that $M_0$ is a closed aspherical manifold, and assume its fundamental group $\G$ is residually finite. Then the characteristic map $\chi$ induces  an isomorphism
 $ \cE(\widehat{M_0} , \whx) \cong  \Comm(\G)$.
\end{thm}

 For both of these previous results, the fact that the groups $\G$ considered are torsion free and the universal group chains  defining  $\fM_{\cP}$ are normal imply that Problem ~\ref{prob-lifting} always has a solution. In addition, the result  of Odden uses the theory of surface homeomorphisms to show that Problem~\ref{prob-surjective} has a solution. We next describe a dynamical approach  to Problems \ref{prob-lifting} and \ref{prob-surjective}.

Section~\ref{sec-dynamics} discusses some dynamical properties of odometer actions. In particular, we recall the notion of topologically free actions, and introduce  the more general  notion  of   a \emph{coherent} odometer action.  Then we have:
 
\begin{thm}\label{thm-injective}  If   $(\fX, \G, \Phi)$  is  an effective and  coherent  odometer, then the action is regular.
\end{thm}
Then as a special case we have:
   
 \begin{cor} \label{cor-eqa}
If an odometer   $(\fX, \G, \Phi)$ is topologically free, then it is regular.
\end{cor}

Combine this result with Proposition~\ref{prop-nilpotentcoherent} to obtain the following general result:  
   \begin{cor} \label{cor-nilp}
If $\G$ is   nilpotent, then an effective odometer action $(\fX, \G, \Phi)$ is regular.
\end{cor}
 
The solution of Problem~\ref{prob-surjective}   requires   deep results from  the topology of manifolds.  We    are  given a class     in the algebraically defined group $\Comm(\fX_0, \G, \Phi, x)$ that maps to $\xi \in \Comm(\G)$,  and must show there  exists  $h_{\xi}$ in the topological group $\Homeo(\fM_{\cP}; \whx)$ so that $\chi([h_{\xi}]) = \xi$. Our    construction of the homeomorphism $h_{\xi}$ uses that  the base manifold $M_0$ is strongly Borel, as in Definition~\ref{def-stronglyborel}. The Borel property says  that   homotopy equivalent manifolds are homeomorphic. This should be contrasted to the approach used by Odden,  which is based on relationship between maps on  the boundary of the hyperbolic disk and homeomorphisms of the disk.

We  also  introduce the class of manifolds called ``ultra Borel'' for which  the isotopy class of the homeomorphism $h_{\xi}$ constructed from a class in $\Comm(\fX_0, \G, \Phi, x)$ is uniquely defined. There are many classes of closed manifolds which satisfy the ultra Borel condition, as discussed in  Section~\ref{sec-classify}.

\begin{thm} \label{thm-surjective}
Assume that $M$ is an ultra Borel connected manifold, and hence is aspherical with torsion-free fundamental group $\G =\pi_1(M,x)$. Let $\fM_{\cP}$ be a solenoidal manifold over $M$ with  a simply connected leaf. 
Then the map 
$\ds 
 \sigma_{\cP,  \whx} \colon \MCG(\fM_{\cP} ; \whx) \to \Comm(\fX_0, \G, \Phi, x)  
$
is an isomorphism. Thus, the map $\chi \colon \MCG(\fM_{\cP} ; \whx)  \to \Comm(\G)$ is an isomorphism onto the image of $\Theta_{\cG}$.
\end{thm}

The final step in the analysis of   $\MCG(\fM_{\cP})$ in Section~\ref{sec-normalizer}  relates it to the group   $ \MCG(\fM_{\cP} ; \whx)$. 
When $\fM_{\cP}$ is homogeneous,   we obtain an explicit calculation of  $\MCG(\fM_{\cP})$. This result  was   shown for the  universal hyperbolic solenoid   by Odden \cite{Odden2005}.  We extend his method to all  homogeneous solenoids.    

To state the general result, consider the group chain model $(X_{\cG}, \G  , \Phi)$ for the monodromy action, as described  in Section~\ref{sec-odometers}, with the profinite groups $\fD_{\cG} \subset \fG_{\cG}$.  Here $\fG_{\cG}$ is isomorphic with the Ellis group of the odometer action, which is the closure of the action $(X_\cG,\G,\Phi)$ in $\Homeo(X_\cG)$, and $\fD_{\cG}$ is the isotropy group of the action of $\fG_{\cG}$ on $X_{\cG}$ at $x$. The space $\fM_{\cP}$ is homogeneous precisely when $\fD_{\cG}$ is the trivial group by the results   \cite{CHL2019a,FO2002}. Then we have:

 \begin{thm}\label{thm-homogeneous}
Let $\fM_{\cP}$ be a homogeneous solenoidal manifold. 
Assume the baseleaf $L_{\whx} \subset \fM_{\cP}$ is simply connected. Then we have a topological isomorphism
\begin{equation}\label{eq-fullprod}
\MCG(\fM_{\cP}) \cong  \fG_{\cG}   \times_{\zeta} \MCG(\fM_{\cP} ; \whx)   \ .
\end{equation}
  \end{thm}
 The action $\zeta$ of $\G$ on  $\MCG(\fM_{\cP} ; \whx)$ is by a leafwise translation, as defined by formula  \eqref{eq-normaction}.

    \begin{cor}\label{cor-homogeneous}
  Let $M_0$ be an ultra Borel   manifold, and assume that  $\fM_{\cP}$ is a homogeneous   solenoid over $M_0$ with a simply connected baseleaf $L_{\whx} \subset \fM_{\cP}$. Then 
  \begin{equation}\label{eq-fullproddyn}
\MCG(\fM_{\cP}) \cong    \fG_{\cG}   \times_{\zeta} \Comm(\fX_0, \G, \Phi, x)  \ .
\end{equation}
  \end{cor}

When $\fM_{\cP}$ is not homogeneous, the isotropy subgroup $\fD_{\cG} \subset \fG_{\cG}$ is not trivial, and hence    the analysis of   $\MCG(\fM_{\cP})$    is more subtle. 
The normalizer group    $\fN_{\cG} = \fN_{\fG_{\cG}}(\fD_{\cG})$   was studied in \cite[Section~2]{DHL2016a}, and the quotient group $\fA_{\cG} = \fD_{\cG} \backslash \fN_{\cG}$ was introduced.   There is a right  natural action of $\fA_{\cG}$ on $\Homeo(\fM_{\cP})$ which descends to an action on $\MCG(\fM_{\cP})$.

   Let  $\ogamma \in \fG_{\cG}$ denote the image of $\gamma$ in the completion $\fG_{\cG}$.
 Define     $\G_{\fN} = \{ \gamma \in \G \mid \ogamma \in \fN_{\cG} \}$.    That is, $\G_{\fN}$ is the subgroup consisting of elements in $\G$ which act on both the left and right on the quotient $X_{\cG} = \fG_{\cG}/\fD_{\cG}$. 
There is a natural action of $\G_{\fN}$ on $\fM_{\cP}$ by baseleaf preserving homeomorphisms, as defined in \eqref{eq-normaction}, which we use to define the quotient space $  \fA_{\cG} \times_{\zeta} \Homeo(\fM_{\cP}, L_{\whx})$ in \eqref{eq-qiotientspace}, where $L_{\whx}$ is the leaf  of $\fM_{\cP}$ containing $\whx$.  The proof of the following result is in Section~\ref{sec-normalizer}:

 \begin{thm}\label{thm-nonhomogeneous}
Let $\fM_{\cP}$ be a solenoidal manifold. 
Assume the baseleaf $L_{\whx} \subset \fM_{\cP}$ is simply connected. Then there is an injection of topological spaces
\begin{equation}\label{eq-tau}
  \omu_{\cG} \colon  \fA_{\cG}  \times_{\zeta} \MCG(\fM_{\cP} ; \whx)    \to \MCG(\fM_{\cP})  \ .
\end{equation}
  \end{thm}
  
      \begin{cor}\label{cor-nonhomogeneous}
       Let $M_0$ be an ultra Borel   manifold with a simply connected baseleaf $L_{\whx} \subset \fM_{\cP}$. Then there is an injection of topological spaces 
  \begin{equation}\label{eq-partialproddyn}
 \omu_{\cG} \colon \fA_{\cG} \times_{\zeta} \Comm(\fX_0, \G, \Phi, x)   \to  \MCG(\fM_{\cP}) \ .
\end{equation}
  \end{cor}

     Section~\ref{sec-examples}   gives a collection of examples of  calculations of $\MCG(\fM_{\cP})$ using the results of this paper.

 The assumption there is a simply connected leaf in $\fM_{\cP}$  implies the associated odometer action $(X_{\cG}, \G, \Phi)$ is regular. Without this assumption, the map $\Theta_{\cP}$ in Diagram~\ref{eq-diagram} may not exist.  Thus, the following problem is interesting, especially when  $\G$ contains a non-trivial torsion subgroup.
  
  \begin{problem}\label{prob-notcoh}
  Let $(\fX, \G, \Phi)$ be an odometer that is not coherent. What is the relation between the dynamical commensurator $\Comm(\fX, \G, \Phi, x)$ and the abstract commensurator $\Comm(\G)$?
  \end{problem}

\section{Solenoidal manifolds}\label{sec-solenoids}

Let $M_0$ denote  a compact connected manifold without boundary, and   $\G = \pi_1(M_0,x_0)$ for   $x_0 \in M_0$. 

Let  $\cG = \{\G = \G_0 \supset \G_1 \supset \G_2 \supset \cdots\}$ be a descending chain of subgroups of finite index.  
The group chain $\cG$ and   $x_0\in M_0$ determine a tower of finite index coverings by closed manifolds:
\begin{equation}\label{eq-presentation2}
\cP_{\cG}  \equiv ~  \left\{ M_0   \stackrel{p_1}{\longleftarrow}    M_1    \stackrel{p_2}{\longleftarrow}  M_2   \stackrel{p_3}{\longleftarrow}      \cdots  \right\} \ . 
\end{equation}
  The inverse limit space 
 \begin{eqnarray} 
\fM_{\cP} & \equiv &  \varprojlim  ~ \{ p_{\ell +1} \colon M_{\ell +1} \to M_{\ell}\} \label{eq-presentationinvlim} \\
& = &  \{(y_0, y_1, \ldots ) \in \fM_{\cP}  \mid p_{\ell +1 }(y_{\ell + 1}) =  y_{\ell} ~ {\rm for ~ all} ~ \ell \geq 0 ~\} ~ \subset \prod_{\ell \geq 0} ~ M_{\ell} ~   \nonumber
\end{eqnarray}
  is given  the relative  topology, induced from the product (Tychonoff) topology, so that $\fM_{\cP}$ is itself compact and connected.  
Given the basepoint $x_0 \in M_0$, for each $\ell > 0$ choose    $x_{\ell +1} \in M_{\ell +1}$  such that $p_{\ell +1}(x_{\ell+1}) = x_{\ell}$. 
Then for  each $\ell \geq 0$, there is a fibration $\whq_{\ell} \colon \fM_{\cP} \to M_{\ell}$, where each fiber   $\fX_{\ell} = \whq_{\ell}^{-1}(x_{\ell})$   a Cantor space. 
  The   manifold $M_0$ is called the base of the tower $\cP_\cG$, and $\whx = (x_0, x_1, x_2, \ldots) \in \fM_{\cP}$   determines a basepoint in the limit.

McCord observed in \cite{McCord1965} that     $\fM_\cP$ has a local product structure, and   the path-connected components of $\fM_\cP$   define a foliation denoted by $\FP$.   That is, $\fM_\cP$ is an  $n$-dimensional  foliated space  in the sense of \cite{MS2006}, for which  the leaves of $\F_\cP$ are the path connected components of $\fM$, and the local transversals to the foliation  are totally disconnected.         
  
  We recall an elementary result about  solenoids.   For $\ell \geq 0$, introduce the truncated presentation
  \begin{equation}\label{eq-presentationell}
\cP_{\ell}  \equiv ~  \left\{ M_{\ell}   \stackrel{p_{\ell+1}}{\longleftarrow}    M_{\ell+1}    \stackrel{p_{\ell+2}}{\longleftarrow}  M_{\ell +2}   \stackrel{p_{\ell +3}}{\longleftarrow}      \cdots  \right\} \ . 
\end{equation}
 
 \begin{lemma}\label{lem-canonical}
 For all $\ell \geq 0$ there is a canonical homeomorphism $\sigma_{\ell} \colon \fM_{\cP_{\ell}} \cong  \fM_{\cP}$.
 \end{lemma}
 \proof
 Given $\why = (y_{\ell}, y_{\ell +1}, y_{\ell +2}, \ldots) \in \fM_{\cP_{\ell}}$ define the shift map on indices, 
 $$\sigma_{\ell}(y_{\ell}, y_{\ell+1}, y_{\ell+2}, \ldots) = (p^{\ell}_0(y_{\ell}), p^{\ell+1}_1(y_{\ell+1}), p^{\ell+2}_2(y_{\ell+2}), \ldots) \in \fM_{\cP} \ , $$
 where $p^{\ell'}_{\ell} = p_{\ell+1} \circ \cdots \circ p_{\ell' -1} \circ p_{\ell'} \colon M_{\ell'} \to M_{\ell}$. 
It is standard that $\sigma_{\ell}$ is a homeomorphism onto.
 \endproof

 \section{The characteristic map}\label{sec-char}

In this section we prove Theorem~\ref{thm-chi}. That is,   given a solenoidal manifold $\fM_\cP$ with presentation $\cP$, we construct the characteristic map $\chi \colon \MCG(\fM_{\cP} ; \whx)  \to \Comm(\G)$.
   
 Fix a metric $d_{\fM}$ on $\fM_{\cP}$ compatible with the inverse limit topology.
A map $H \colon \fM_{\cP} \times [0,1] \to \fM_{\cP}$ is said to be an $\e$-homotopy if the trace of each point is $\e$-bounded. That is, for all $\why \in \fM_{\cP}$ we have $d_{\fM}(H(\why, s), H(\why, t)) < \e$ for all $s,t \in [0,1]$.  
We use a fundamental   result   of Rogers and Tollefson:
    \begin{thm}\cite[Corollary~5.2]{RT1972b}\label{thm-RT}
     Let $M_0$ be a closed connected manifold, and  $\fM_{\cP}$ be a solenoidal manifold over $M_0$ defined by a presentation $\cP$.  Given a self-homeomorphism  $h \colon \fM_{\cP} \to \fM_{\cP}$ and $\e > 0$,  then  $h$ is $\e$-homotopic to a continuous map   induced by a \emph{shuffle map} of $\cP$. That is, there exists    an increasing sequence   $i_{\ell} \geq \ell$, and continuous maps $\varphi_{i_\ell} \colon M_{i_{\ell}} \to M_{\ell}$  so that the ``zig-zag'' diagram \eqref{eq-shuffle} commutes
 \begin{align} \label{eq-shuffle}
 \xymatrix{ M_{0}   & \ar[l] M_{i_0} \ar[ld]^{\varphi_{i_0}}& \ar[l] M_{i_1} \ar[ld]^{\varphi_{i_1}} & \ar[l] M_{i_2} \ar[ld]^{\varphi_{i_2}}  & \ar[l] \cdots \\
M_0 & \ar[l] M_1  & \ar[l] M_2 & \ar[l] M_3   & \ar[l]\cdots } 
\end{align}
and the self-map   $\whvp \colon \fM_{\cP} \to \fM_{\cP}$ induced by the collection $\{\varphi_{i_{\ell}} \mid \ell \geq 0\}$ is $\e$-homotopic to $h$. Moreover, for each $\ell \geq 0$, the map $\varphi_{i_{\ell}}$ induces a monomorphism between the fundamental groups of $M_{i_{\ell}}$ and $M_{\ell}$ whose image has finite index.
    \end{thm}
\proof  
 For each $\ell \geq 0$, there is the proper covering map $p_0^{\ell} = p_{\ell} \circ \cdots \circ p_1 \colon M_{\ell} \to M_0$, and the projection map 
  $\whq_{\ell} \colon \fM_{\cP} \to M_{\ell}$  onto the factor $M_{\ell}$ in \eqref{eq-invlimit}.  
  Let $\whx \in \fM_{\cP}$ be a choice of basepoint, and  set $x_{\ell} = \whq_{\ell}(\whx) \in M_{\ell}$  and set $\G_0 = \pi_1(M_0, x_0)$. For each $\ell > 0$,  $\G_{\ell}$ is the image of $(p_0^{\ell})_{\#} \colon \pi_1(M_{\ell}, x_{\ell}) \to \G_0$, so $\cG = \{\G_{\ell} \mid \ell \geq 0 \}$ is a group chain in $\G_0$.
  
  Let $d_{M_0}$ denote a Riemannian metric on $M_0$ and let $\delta_0 > 0$ be sufficiently small so  that, for all $y \in M_0$, the closed metric ball $B_{M_0}(y, \delta_0) \subset M_0$ of radius $\delta_0$ is convex. In particular, $\delta_0$ is less than the injectivity radius of $M_0$. It follows that for each $y \in M_0$ and $y' \in  B_{M_0}(y, \delta_0)$ there is an isomorphism $\iota_{y',y} \colon \pi_1(M_0, y') \cong \pi_1(M_0, y)$ induced by concatenation with a geodesic path  in $B_{M_0}(y, \delta_0)$  from $y'$ to $y$.

   Finally, since the projection $\whq_0 \colon \fM_{\cP} \to M_0$  is uniformly continuous, we can choose  $\e >0$   sufficiently small so that  $\whq_0(B_{\fM}(\why, \e)) \subset B_{M_0}(y, \delta_0/3)$ for all $\why \in \whq_0^{-1}(y)$, where $B_{\fM}(\why,\e)$ is a ball around $\why \in \fM_\cP$ with respect to the metric $d_\fM$.

  Given $h \in \Homeo(\fM_{\cP}; \whx)$, we first define $\chi(h) \in \Comm(\G)$. We then   show that $\chi(h)$ is independent of choices, and depends only on the basepoint fixed homotopy class of $h$.
   
Let $h \colon \fM_{\cP} \to \fM_{\cP}$, then by Theorem~\ref{thm-RT} there exists an increasing  sequence $i_{\ell} \geq \ell$, $\ell \geq 0$,  and maps $\varphi_{i_\ell} \colon M_{i_{\ell}} \to M_\ell$  so that the shuffle diagram \eqref{eq-shuffle} commutes, and for  
 the map  $\whvp \colon \fM_{\cP} \to \fM_{\cP}$ defined by the maps $\{\varphi_{i_{\ell}} \mid \ell \geq 0\}$, we have that $h$ and $\whvp$ are $\e$-homotopic.  Let $\why = \whvp(\whx)$ and set $y_{\ell} = \whq_{\ell}(\why) \in M_{\ell}$.
Then we have the diagram:
  \begin{align} \label{eq-defchi}
 \xymatrix{  \pi_1(  M_{i_0} , x_{i_0}) \ar[d]_{(p_0^{i_0})_{\#}} \ar[r]^{(\vp_{i_0})_{\#}} &   \pi_1(M_{0} ,y_0)  \ar[r]_{\cong}^{\iota_{y',y}}  &  \pi_1(M_0 , x_0)   \\
\pi_1(M_0 , x_0) &  &   } 
\end{align}
where $(p_0^{i_0})_{\#} \colon  \pi_1(  M_{i_0} , x_{i_0})  \to \G_{i_0} \subset \G_0$ is an isomorphism onto its image $\G_{i_0}$.
Let $h_{\#} \colon \G_{i_0} \to \G_0$ be  the commensurator defined by the inverse of $(p_0^{i_0})_{\#}$ followed by   $\iota_{y',y} \circ (\vp_{i_0})_{\#}$.  Define  $\chi(h) \in \Comm(\G_0)$ to be the commensurator class of $h_{\#}$.
 
Next, we show that $\chi(h)$ is independent of the choice of an approximation to $h$ by an induced map. 
Suppose there are given two sequences $\{\varphi_{i_\ell} \colon M_{i_{\ell}} \to M_{\ell} \mid \ell \geq 0\}$ and $\{\psi_{j_\ell} \colon M_{j_{\ell}} \to M_{\ell} \mid \ell \geq 0\}$,   so that the shuffle diagram \eqref{eq-shuffle} commutes for each.
  Moreover,   assume that   the induced maps   $\whvp \colon \fM_{\cP} \to \fM_{\cP}$ and  $\whpsi \colon \fM_{\cP} \to \fM_{\cP}$   are both $\e$-homotopic to $h$. Thus $\whvp$ and $\whpsi$ are $2\e$-homotopic.

    Let $\why = \whvp(\whx)$ then set  $y_{k} = \whq_{k}(\why) \in M_{k}$ for   $k \geq 0$. Likewise, let
      $\whz = \whpsi(\whx)$ and set  $z_{k} = \whq_{k}(\whz) \in M_{k}$ for   $k \geq 0$.
      Choose an index $k_0 > \max\{i_0, j_0\}$. 
  Then we have two maps:
 \begin{eqnarray*}
\Psi_{\whvp} \colon & \fM_{\cP} \stackrel{\whq_{k_0}}{\longrightarrow} M_{k_0} \stackrel{p^{k_0}_{i_0}}{\longrightarrow} M_{i_{0}} \stackrel{\vp_{i_0}}{\longrightarrow}   M_0\\
\Psi_{\whpsi}  \colon & \fM_{\cP} \stackrel{\whq_{k_0}}{\longrightarrow} M_{k_0} \stackrel{p^{k_0}_{j_0}}{\longrightarrow} M_{j_{0}} \stackrel{\psi_{j_0}}{\longrightarrow}   M_0   ~ .
\end{eqnarray*}
Now, given $\alpha \in \pi_1(M_{k_0}, x_{k_0})$ choose a  closed curve $\gamma$ in $M_{k_0}$ with basepoint $x_{k_0}$ representing $\alpha$.

Let $\wtgamma$ be the lift of $\gamma$ to a curve in   the leaf $L_{\whx}$ in $\fM_{\cP}$ starting at $\whx$ and which covers $\gamma$.

The images   $\Psi_{\whvp}(\wtgamma)$ and $\Psi_{\whpsi}(\wtgamma)$ in   $M_0$  of the curve  $\wtgamma$ in   $\fM_{\cP}$ then lie within $2\delta_0/3$ of each other. By our choice of $\delta_0$,  the  images of the  curves   lie in a convex neighborhood of each other, and hence are homotopic. Thus, they determine the same equivalence class in $\pi_1(M_0, x_0)$. This implies that the definition of $\chi(h)$ on elements of $\G_{k_0}$ agrees using either approximation $\whvp$ or $\whpsi$, and $\chi(h)$ does not depend on the choice of an approximating homotopy.

Finally, let $\{h_t \colon \fM_{\cP} \to \fM_{\cP} \mid 0 \leq t \leq 1\} \subset \Homeo(\fM_{\cP};  \whx)$ be an isotopy. Let $\e > 0$ be as above, then there exists $\lambda > 0$ such that   $| s-t| < \lambda$ implies $d_{\fM}[h_s , h_t] < \e$, where $d_{\fM}[\cdot , \cdot]$ denotes the sup-norm on homeomorphisms induced by $d_{\fM}$.  Then for $n \gg 0$ so that $1/n < \lambda$ we form the sequence of isotopic homeomorphisms $\{h_{i/n} \mid 0 \leq i \leq n\}$ with $d_{\fM}[h_{i/n} , h_{(i+1)/n}] < \e$.
 As above, for each $i$ choose an induced tower map $\whvp_i$ which approximates $h_{i/n}$.  Then $\whvp_i$ and $\whvp_{i+1}$ are $3\e$-homotopic, and so $\chi(\whvp_i)$ and $\chi(\whvp_{i+1})$ determine the same class in $\Comm(\G)$. It follows that 
 $\chi(\whvp_0)$ and $\chi(\whvp_n)$ determine the same class in $\Comm(\G)$.  This shows that $\chi(h)$ depends only on the isotopy class of $h$ hence $\chi$ descends to  a map on $\MCG(\fM_{\cP}, \whx)$. This completes the proof of Theorem~\ref{thm-chi}.
 \endproof
 
 \begin{remark}
{\rm
 The conclusion of Theorem~\ref{thm-chi} can be reformulated in terms of pro-categories.
Theorem~\ref{thm-RT} implies that a homeomorphism $h \colon \fM_{\cP} \to \fM_{\cP}$ induces a map $\{ \vp_{i_{\ell}} \}$ in the pro-category of spaces. By the local contractibility of the base $M_0$ the pro-map induces a pro-map in the category of groups, which induces the map $\chi$. This is the approach followed in  \cite{BeringStudenmund2024}.
} 
\end{remark}

\section{Algebraic odometers and dynamical commensurators} \label{sec-odometers}
 
We give the proof of Theorem \ref{thm-charmap}, after first recalling some of the basic notions of odometers defined by group chains \cite{CortezPetite2008,CortezMedynets2016}.  An \emph{action} of a group $\G$ on a topological space $\fX$ is given by a homomorphism $\Phi:\G \to \Homeo(\fX)$. This gives an assignment, denoted by the same letter
$$\Phi: \G \times \fX \to \fX: (g,y) \mapsto \Phi(g)(y) \ , \quad  g \cdot y : = \Phi(g)(y) \ .$$

The action  is \emph{effective} if  $\Phi \colon \G \to \Homeo(\fX)$ is injective.

An action $\Phi: \G \times \fX \to \fX$ is \emph{minimal} if for every $y \in \fX$ the orbit $O(y) = \{g \cdot y : g \in \G\}$ is dense in $\fX$. Let $d_{\fX}$ be a metric on $\fX$ compatible with the topology. An action $\Phi$ is \emph{equicontinuous} if for every $\epsilon > 0$ there exists $\delta > 0$ such that for all $y, z \in \fX$ with $d_{\fX}(y,z) < \delta$ and all $g \in \G$ we have $d_{\fX}(g \cdot y, g\cdot z) < \epsilon$.  An isometric action is equicontinuous.

\begin{defn}\label{def-odometer}
Let $\G$ be a finitely generated group and $\fX$ be a Cantor space. Let $d_\fX$ be a metric on $\fX$ compatible with the topology. We say that an action   
$\Phi \colon \G \times \fX \to \fX$ is an \emph{odometer} if it is minimal and equicontinuous.
\end{defn}
These general actions are called \emph{subodometers} in   Cortez \cite{CortezPetite2008}, and elsewhere called  \emph{generalized odometers}. We use the briefer terminology ``odometer''  for  all such actions.

 Let $\cG = \{ \G = \G_0 \supset \G_1 \supset \G_2  \supset \cdots\}$ be a group chain in a finitely generated group  $\G$.  
 We associate to $\cG$ a left odometer action of $\G$ on a Cantor space $X_{\cG}$ as follows.
  
For $\ell \geq 0$ let   $X_{\ell} = \G/\G_{\ell}$ as a left  $\G$-space.   The inclusion   of cosets induces a projection map 
$p_{\ell+1} \colon X_{\ell+1} \to X_{\ell}$ which commutes with the left action of $\G$. Introduce the inverse limit space
\begin{equation}\label{eq-invcosetlim}
X_{\cG}   \equiv    \lim_{\longleftarrow} ~ \{ p_{\ell +1} \colon X_{\ell +1} \to X_{\ell}  \mid \ell \geq 0 \}    \subset    \prod_{\ell \geq 0} ~ X_{\ell}     \ .
\end{equation}
A point $y = (y_0, y_1, \ldots ) \in X_{\cG}$ if $p_{\ell +1}(y_{\ell +1}) =  y_{\ell}$ for all $\ell \geq 0$.
The   $\G$-action on  each factor $X_{\ell}$ induces the  $\G$-action   on $X_{\cG}$. That is, $g \in \G$ acts coordinate-wise, so  $g \cdot y = (g\cdot y_0, g\cdot y_1, \ldots)$.

Let $d_{\ell}$ denote the indicator metric on $X_{\ell}$, so $d_{\ell}(y_{\ell} ,z_{\ell}) = 0$ if $y_{\ell}=z_{\ell}$, and $d_{\ell}(y_{\ell},z_{\ell})=1$ for $y_{\ell} \ne z_{\ell} \in X_{\ell}$. This metric is invariant under the $\G$-action on $X_{\ell}$. Define a $\G$-invariant metric $d_X$ on $X_{\cG}$ by setting, for $y,z \in X_{\cG}$, 
$\ds d_X(y,z) = \sum_{\ell \geq 0} 2^{-\ell} d_{\ell}(y_{\ell}, z_{\ell})$. 
The action $\Phi \colon \G \times X_{\cG} \to X_{\cG}$  is isometric.

 Let    $e \in \G$ denote the identity, and set 
\begin{equation}\label{def-eql}
U_{\ell} = \{ y \in X_{\cG} \mid y_0 =e, \ldots, y_{\ell -1} = e\} \ ,
\end{equation}
 where we denote also by $e$ the coset of the identity in $X_\ell$.
The collection $\{U_{\ell} \mid \ell \geq 0\}$ forms a clopen neighborhood basis of the basepoint $x = (e,e,\ldots) \in X_{\cG}$, where $\cap_{\ell > 0} \ U_{\ell} = \{x\}$.  The collection of $\G$-translates of this neighborhood basis forms a basis of clopen sets for the topology on $X_{\cG}$, which is a Cantor space.
The action of $\G$ on   $X_{\ell}$  is transitive, so the action $\Phi \colon \G \times X_{\cG} \to X_{\cG}$  is minimal and equicontinuous, hence    is  an odometer.

 A clopen set  $U \subset X_{\cG}$ is \emph{adapted} if $U \ne \emptyset$, and  for all $\gamma \in \G$,  if  $\gamma \cdot U \cap U \ne \emptyset$ then $\gamma \cdot U = U$. It follows that 
   $\G_U = \{\gamma \in \G \mid \gamma \cdot U = U\}$   is a subgroup of finite index in $\G$. Observe that $\G_{U_{\ell}} = \G_{\ell}$.

  Associated to an odometer $\Phi \colon \G \times X_{\cG} \to X_{\cG}$ is a profinite group $\fG_{\Phi}$ and an  action $\oPhi \colon \fG_{\Phi} \times X_{\cG} \to X_{\cG}$.  The group $\fG_{\Phi}$ can be defined as the closure of the image of $\Phi \colon \G \to \Homeo(X_{\cG})$ in the uniform topology, which is called the Ellis semi-group for the action \cite{Ellis1960,Auslander1988}. In our case, $\fG_{\Phi}$ is a group since the action is equicontinuous. There is an alternate   description of $\fG_{\Phi}$   as follows, see \cite{DHL2016a} for details.
  
    The \emph{normal core} of   $H \subset \G$ is the largest normal subgroup  $C(H) \subset H$. 
  For each subgroup $\G_{\ell}$ its normal core is given by
  $C_{\ell} = C(\G_{\ell}) = \bigcap_{g \in \G} g \G_{\ell} g^{-1} $
  which is identified with the kernel of the action map $\Phi_{\ell} \colon \G_{\ell} \to \Aut(\G/\G_{\ell})$. 
  Let $G_{\ell} = \G/C_{\ell}$ then the inverse limit 
  \begin{equation}\label{eq-invgrplim}
\fG_{\cG}   \equiv    \lim_{\longleftarrow} ~ \{ p_{\ell +1} \colon G_{\ell +1} \to G_{\ell}  \mid \ell \geq 0 \}    \subset    \prod_{\ell \geq 0} ~ G_{\ell}     
\end{equation}
is a profinite group with a left  isometric  $\fG_{\cG}$-action on $X_{\cG}$ denoted by $\oPhi \colon \fG_{\cG} \times X_{\cG} \to X_{\cG}$. The group $\fG_\cG$ is isomorphic to $\fG_\Phi$ (see \cite{DHL2016a}.) As the    action of $\G$ on $X_{\cG}$ is minimal,  the   action $\oPhi$ of $\fG_{\cG}$ on $X_{\cG}$ is transitive. 
There is a natural map $\iota_{\cG} \colon \G \to \fG_\cG: g \mapsto (g,g,\ldots)$, which is   induced from  the quotient maps $\G \to  G_{\ell}$, and we denote also by $g$ the coset of $g$ in $G_\ell$.
\begin{lemma}\label{lem-injective}
 $\iota_{\cG} \colon \G \to \fG_{\cG}$ is injective if and only if $\bigcap_{\ell > 0} C_{\ell} = \{e\}$, where $e$ is the identity in $\G$.
\end{lemma}
The map  $\iota_{\cG}  \colon \G \to \fG_{\cG}$   corresponds to the action map $\Phi \colon \G \to \Homeo(X_{\cG})$, so 
the action  is effective  if  and only if  $ \bigcap_{\ell \geq 0} C_{\ell} = \{e\}$ is the trivial group. In particular, if the action is effective then  $\G$ is residually finite.

Let $\fD_{\cG,x} = \{ \whg \in \fG_{\cG} \mid \oPhi(\whg)(x) = x\}$ be the isotropy subgroup at the basepoint $x \in X_{\cG}$. We omit the subscript $x$ when it   is clear from the context. There is a  $\fG_{\cG}$-equivariant isomorphism $X_{\cG} \cong \fG_{\cG}/\fD_{\cG}$. 

For a subgroup $H \subset \G$,  let $\oH \subset \fG_{\cG}$ denote its closure in $\fG_{\cG}$. In particular, the closure $\oG_{\ell}$ acts transitively on $U_{\ell}$, hence $U_{\ell} = \oG_{\ell}/\fD_{\cG}$ for all $\ell \geq 0$, and so $\bigcap_{\ell} \oG_{\ell} = \fD_{\cG}$.

There is an alternate definition of $\fD_{\cG}$ as an inverse limit. 
Consider the quotient $D_{\ell} = \G_{\ell}/C_{\ell}$. Since $C_{\ell} $ is a normal subgroup in $G$, and hence also normal in $\G_{\ell}$, the quotient $D_{\ell}$ is a subgroup of  $G_{\ell} = \G/C_{\ell}$. For $\ell' > \ell$ the inclusion map  $\G_{\ell'} \to \G_{\ell}$ induces  a homomorphism $\theta^{\ell'}_{\ell} \colon  \G_{\ell'}/C_{\ell'} \to \G_{\ell}/C_{\ell}$.
By  \cite[Theorem~4.4]{DHL2016a} we then have a topological isomorphism
\begin{equation}\label{eq-discriminant}
\fD_{\cG} \cong \lim_{\longleftarrow}\, \left \{\theta^{\ell'}_{\ell} \colon \G_{\ell'}/C_{\ell'} \to \G_{\ell}/C_{\ell} \right\} \ .
\end{equation}

The subgroup $\fD_{\cG}$ is \emph{totally not normal} in $\fG_{\cG}$. That is, the largest   subgroup of  $\fD_{\cG}$ which is normal in $\fG_{\cG}$ is the trivial group by \cite[Corollary~5.4]{DHL2016a}.  Consequently, the      action ${\rm Adj} \colon  \fD_{\cG} \to \Homeo(X_{\cG};  x)$ induced from the adjoint action of  $\fD_{\cG}$ on $\fG_{\cG}$ is an injection.

  We next define some  notions of equivalence for odometer actions.  
    \begin{defn} \label{def-isomorphism}
Odometer actions $(\fX_1, \G_1, \Phi_1)$ and $(\fX_2, \G_2, \Phi_2)$ are said to be \emph{isomorphic}  if there is a homeomorphism $h \colon \fX_1 \to \fX_2$ and group isomorphism $\Theta \colon \G_1 \to \G_2$ such that 
\begin{equation}\label{eq-isomorphism}
\Phi_1(g) = h^{-1} \circ \Phi_2(\Theta(g)) \circ h   \in   \Homeo(\fX_1) \   \textrm{for  all} \ g \in \G_1 \ .
\end{equation}
The map  $h$ is said to be an \emph{isomorphism} when    $\fX_1 = \fX_2$ and $\G_1 = \G_2$. In this case,   $\Theta \in \Aut(\G)$ can be viewed as a ``uniform time change'' for the action.

The isomorphism $h$ is an \emph{automorphism}   if $\fX_1 = \fX_2$, $\G_1 = \G_2$, and $\Theta \colon \G \to \G$ is the identity.
 \end{defn}    
 
When $\fX_1 = \fX_2$, we can then define   the following group invariants of the action.
\begin{defn}\label{def-normalizer}
Let $(\fX, \G, \Phi)$ be an odometer action. Define the isomorphism  group  
\begin{equation}\label{eq-normalizer}
\fI(\fX, \G, \Phi) = \{ h \in \Homeo(\fX)  \mid h ~ {\rm is \ an \ isomorphism \ of \ the \ action}\} \ .
\end{equation}
  The  automorphisms   form a normal subgroup
\begin{equation}\label{eq-auto}
\fI_0(\fX, \G, \Phi) = \{ h  \in \Homeo(\fX)  \mid h ~ {\rm is \ an \ automorphism \ of \ the \ action}\} \ .
\end{equation}
The  quotient group 
\begin{equation}\label{eq-rep}
\fQ(\fX, \G, \Phi) = \fI(\fX, \G, \Phi)/\fI_0(\fX, \G, \Phi) \  
\end{equation}
then has a representation $\Theta \colon \fQ(\fX, \G, \Phi) \to \Aut(\G)$.
\end{defn}

For the case when $\G = \mZ^n$, the map $\Theta$ has values in ${\rm GL}(n,\mZ)$, and  the image of $\Theta$ is called the ``linear representation group'' by Cabezas and Petite   \cite{CabezasPetite2024}.  As they remark, the calculation of the range of $\Theta$ is  a difficult open problem already for the case $\G = \mZ^n$ for $n > 1$, except in special cases.  We extend the concept of a linear representation group to  the group of germinal equivalences. 

First, we recall a basic notion.

\begin{defn} \label{def-RE}
Odometers  $\Phi \colon \G \times \fX \to \fX$ and  $\Phi' \colon \G' \times \fX' \to \fX'$ are \emph{return equivalent} if there exist adapted sets $U \subset \fX$ and $U' \subset \fX'$ and a homeomorphism $h_U \colon U \to U'$ that induces an isomorphism $\Theta_U \colon \cH_U \to \cH_V$ between the groups
 \begin{eqnarray*}
\cH_U & = & {\rm Image}\{\Phi_U \colon \G_U \to  \Homeo(U)\} \\
\cH'_{U'} & = &  {\rm Image}\{\Phi'_{U'} \colon \G'_{U'} \to  \Homeo(U')\}.
\end{eqnarray*}
\end{defn}
The notion of \emph{return equivalence} of odometers  corresponds to the notion of \emph{Morita equivalence}.

Note that the   map $\Phi_U \colon \G_U \to \cH_U$ may have kernel, and likewise for $\Phi'_{U'}$. 
 Thus,  a homeomorphism $h_{U} \colon U \to U'$ which induces an isomorphism between   $\cH_U$ with $\cH'_{U'}$ need not induce an isomorphism between  $\G_U$ and $\G'_{U'}$.

The group of germinal equivalences is defined using the germs of return equivalence maps.

    \begin{defn}\label{def-dyncomm2}
   Let  $(\fX, \G, \Phi)$   be an odometer action, and $x \in \fX$. 
   \begin{enumerate}
\item  A \emph{dynamical commensurator} at $x$  is a homeomorphism $h_U \colon U \to V$, where $U$ and $V$ are adapted subsets with $x \in U \cap V$, such that $h_U(x) = x$, and $h_U$ induces an isomorphism $\Theta_U \colon \cH_U \to \cH_V$.
\item   Commensurators $h_U \colon U \to V$ and $h_{U'} \colon U' \to V'$ are equivalent if there exists an adapted set $U''$ with $x \in U'' \subset U \cap U'$ such that $h_U|U'' = h'_{U'}|U''$. We   write $(h_U,U,V) \comm (h'_{U'},U',V')$.
\item   The \emph{dynamical commensurator group} of $(\fX, \G, \Phi)$ at $x$ is the set of germs of dynamical commensurators,
$$
\Comm(\fX, \G, \Phi, x) = \{  h_U \colon U \to V  \mid x \in U \cap V ,  \ h_U \  {\rm is \ a \ commensurator}  \} \slash  \comm  $$
The equivalence class of $h_U \colon U \to V$ is denoted by $[h_U,U,V]$.
\end{enumerate}
\end{defn}

The   dynamical commensurator group is the germinal analog of the group $\fR(\fX, \G, \Phi)$ defined in \eqref{eq-rep}. The definition of the representation $\Theta$ above does not directly extend, as the domain is the group $\G_U$ which is only well-defined modulo the choice of the adapted neighborhood $x \in U$. The content of Theorem \ref{thm-injective} can be considered as making a well-defined germinal analog $\Theta_{\cG}$  for $\Theta$.

  The  motivation for  Definition~\ref{def-dyncomm2}  comes from the following result and its corollary:
\begin{thm}\cite{CHL2019a} \label{thm-inducedmap}
 Suppose that $\fM_{\cP}$ and $\fM'_{\cP'}$ are  homeomorphic solenoidal manifolds with associated group chains $\cG$ and $\cG'$, and corresponding presentations $\cP$ and $\cP'$.   Then their monodromy odometers  $\Phi \colon \G \times X_{\cG} \to X_{\cG}$ and $\Phi' \colon \G' \times X'_{\cG'} \to X'_{\cG'}$ are \emph{return equivalent}. 
\end{thm}

   For a homeomorphism $h \in \Homeo(\fM_{\cP};  \whx)$ the return equivalence induced by $h$ is defined using the holonomy along leafwise paths that start and return in a foliation chart at $\whx$. It follows that the germ of the return equivalence is unchanged by a leafwise homotopy which fixes $\whx$. We thus obtain:
 
 \begin{cor}\label{cor-chi}
 Let $\fM_{\cP}$ be a solenoidal manifold, with monodromy action $(X_{\cG}, \G, \Phi)$.   For $\whx \in \fM_{\cP}$, there is a well-defined map $\sigma_{\cP,  \whx} \colon \MCG(\fM_{\cP} ; \whx) \to \Comm(X_{\cG}, \G, \Phi, x)$.
\end{cor}

Theorem \ref{thm-inducedmap} and Corollary \ref{cor-chi} together show the statement of Theorem \ref{thm-charmap}.

\section{Coherent odometer actions} \label{sec-dynamics}

In this section, we  develop hypotheses on a group chain $\cG$ that are sufficient to prove the existence  of the  map $\Theta_\cG: \Comm(X_{\cG}, \G, \Phi, x) \to \Comm(\G)$ in Diagram \eqref{eq-diagram}, as will be shown in    Section~\ref{sec-lifting}.
  We also give conditions on a group chain which are used in Section~\ref{sec-classify}  to construct homeomorphisms corresponding to commensurators.

An action  $(\fX,\G,\Phi)$ is   said to be  \emph{free} if for all $x \in \fX$ and $\gamma \in \G$,   $\gamma \cdot x = x$ implies that $\gamma = e$,   the identity of the group. 
Let $\Fix(\gamma) = \{x \in \fX \mid \gamma \cdot y = y \}$, and define the \emph{isotropy set}
\begin{equation}\label{eq-isotropy}
 \Iso(\Phi) = \{ y \in \fX \mid \exists ~ \gamma \in \G ~ , ~ \gamma \ne e ~, ~ \gamma \cdot y = y   \} = \bigcup_{e \ne \gamma \in \G} \ \Fix(\gamma) \ . 
\end{equation}
 Recall that a set $X$ is \emph{meager}  if it is a countable union of subsets whose closures have empty interior. A set is \emph{nowhere dense} if its closure has empty interior, so a meager set is equivalently a set which is a countable union of nowhere dense sets. The Cantor space $\fX$ is a Baire space, so   a meager set has empty interior.
   
  \begin{defn} \cite{Boyle1983,BoyleTomiyama1998} \label{def-topfree}
  $(\fX,\G,\Phi)$ is said to be \emph{topologically free}  if  $\Iso(\Phi) $ is meager in $\fX$. 
 \end{defn}
 
Note that if   there exists a non-identity element $\gamma \in \G$ such that $\Fix(\gamma)$ has interior, then the action is not topologically free. 
A topologically free action is always effective, but an effective action may or may not be topologically free. Examples of effective actions which are not topologically free are abundant, including, for instance, the actions of weakly branch groups \cite{Nekrashevych2005}. 
 
 Now suppose that $(X_{\cG}, \G , \Phi)$ is an odometer obtained from the group chain $\cG = \{\G_{\ell} \mid \ell \geq 0\}$. The clopen sets  $\{U_{\ell} \mid \ell \geq 0\}$   as defined by \eqref{def-eql} are adapted with $\G_{U_{\ell}} = \G_{\ell}$, 
 and  $x = \bigcap_{\ell > 0} U_{\ell} \in X_{\cG}$ is the base point.  Set
  \begin{equation}\label{eq-ker}
\ker (\cG) = \bigcap_{\ell > 0} \ \G_{\ell} \ . 
\end{equation}
  The following criteria for topological freeness will be used in Section~\ref{sec-classify}.
\begin{lemma}\label{lem-topfree}
If $\ker (\cG) = \{e\}$ then the action $(X_\cG,\G,\Phi)$ is topologically free.
\end{lemma}
\proof
Note that if $\gamma \cdot x = \gamma' \cdot x$ then $\gamma^{-1} \gamma' \in \G_{\ell}$ for all $\ell > 0$ hence $\gamma^{-1} \gamma'  \in \ker(\cG)$ so $\gamma = \gamma'$. Also note that by continuity of the action $\Fix(\gamma)$ has empty interior if and only if the closure of $\Fix(\gamma)$ has empty interior.

Now let $\gamma' \in \G$ such that  that there exists a non-empty open set  $V \subset \Fix(\gamma')$. Then there exists $\gamma \in \G$ with $\gamma  \cdot x \in V$. Then $\gamma' \cdot \gamma \cdot x = \gamma  \cdot x$, and so $\gamma^{-1} \gamma' \gamma = e$ hence $\gamma' = e$. It follows that each $\Fix(\gamma)$ has no interior for $\gamma \ne e$, hence $ \Iso(\Phi)$ is meager.
\endproof

We introduce another property of odometer actions, that is weaker than the topological free property.
\begin{defn}\label{def-coherent}
   An odometer  $(\fX, \G, \Phi)$  is \emph{coherent} if for every adapted set $U \subset \fX$, the restricted holonomy action map $\Phi_U \colon \G_U \to Homeo(U)$ has finite kernel.
\end{defn}

 \begin{lemma}\label{lem-coherentqa}
A topologically free  odometer  $(\fX, \G, \Phi)$ is coherent.
Conversely, if $\G$ is torsion-free and the action is  coherent, then the action is topologically free.
\end{lemma}
\proof
If $(\fX, \G, \Phi)$ is topologically free, then for every adapted set $U \subset \fX$, the restricted holonomy action map $\Phi_U \colon \G_U \to Homeo(U)$ has trivial kernel. Conversely, assume that the action   is  coherent, and let $\gamma \in \G$ be such that there exists a non-empty open set $V \subset \Fix(\gamma)$. Choose an adapted set $U \subset V$, then $\gamma \in \G_U$ and  $\Phi_U(\gamma) = Id$. By assumption, the kernel of $\Phi_U$ is a torsion subgroup of $\G_U$ hence a torsion subgroup of $\G$. Thus, $\gamma$ is the identity, and so $\Iso(\Phi)$ is a meager set. 
\endproof

We illustrate these concepts with    examples.

\begin{ex}\label{ex-product-lqa-coherent}
{\rm The first example   is that of an action which is effective and coherent, but not topologically free.  It is   a special case of \cite[Example~7.6]{DHL2016a}.

Let $(\fX_0, \G_0, \Phi_0)$ be a free odometer. For example, this can be a standard odometer action of $\G_0 = \mZ$ on the profinite completion $\widehat{\mZ}$.
Let $H$ be a non-trivial finite simple group, and $K \subset H$ a proper non-trivial subgroup.
Let $\G = \G_0 \times H$ be the product group, and let $\fX =  \fX_0 \times H/K$ be the product of $\fX_0$ with the finite space $H/K$. Let $\Phi \colon \G \times \fX \to \fX$ be the product of the   action of   $\G_0$ on the factor $\fX_0$, an the left action of  $H$ on the factor $H/K$.

For each $hK \in H/K$, let $U_h = \fX_0 \times \{hK\}$ which is an adapted clopen set in $\fX$. The subgroup $\G_0$ acts freely on each $U_h$. The subgroup $\{e\} \times K$ fixes every point in the set $U_e = \fX_0 \times \{e K \} $, and for any $h \in H$, the conjugate $\{e\} \times hKh^{-1}$ fixes  every point in $U_h$, which implies that the action of $\G$ on $\fX$ is not topologically free. The intersection $\bigcap_{h \in H} hKh^{-1}$ is trivial, thus the subgroup $H$ acts effectively on $H/K$ and hence $\G$ acts effectively on $\fX$. 
The action of  $\G_{U_h} = \G_0 \times hKh^{-1}$ preserves $U_h$ and the kernel of the action map $\Phi_{U_h} \colon \G_{U_h} \to \Homeo(U_h)$ is the finite group $\{e\} \times hKh^{-1}$. It follows that this action is effective and coherent.
}
\end{ex}

\begin{ex}\label{ex-weakly-branch}
{\rm
For the second example,  suppose that $\G$ is a weakly branch group (see for example \cite{LavrenyukNekrashevych2002,Nekrashevych2005}).
Let  $\fX$ be the boundary of the tree $T$ on which the group acts, then the   action of $\G$ on $\fX$ is  an effective odometer. Weakly branch groups are defined as subgroups of ${\Aut}(T)$, and so their action is effective. A property of weakly branch actions is that for any adapted set $U \subset \fX$ there is an element $g_U \in \G$ such that $g_U|(\fX \backslash U)$ is the identity. It follows that for any adapted $U$ there is an infinite number of elements in $\G$ which are trivial on $U$, and so the action of $\G$ on $\fX$ is not coherent. 
}
\end{ex}

Example \ref{ex-product-lqa-coherent} illustrates why the terminology ``coherent'' is used. For a coherent action $\Phi$, the restriction of the action $\Phi$ to  a clopen set   $U \subset \fX$ is linked, up to finite ambiguity, to the action on the complementary clopen set $\fX \setminus U$. This is contrasted with Example~\ref{ex-weakly-branch}, where the actions on different branches of the trees need not be linked in any way.
 In contrast to the examples of weakly branch groups and their actions, we have:

  \begin{prop}   \label{prop-nilpotentcoherent}
 If $\G$   is  nilpotent, then  an  effective odometer       $(\fX, \G, \Phi)$     is   topologically free.  
 \end{prop}
 \proof

Let $U \subset \fX$ be an adapted set and let $h\in \G_U$ such that $\Phi_U(h)$ is the identity on $U$. We claim that $\Phi(h)$ acts as the identity on $X_{\cG}$.  We proceed by induction on the upper central series
$$\{e\} = \G_{(0)} \subset \G_{(1)} \subset \cdots \subset \G_{(k)} = \G$$
So $\G_{(1)}$ is the center of $\G$, and $\G_{(i+1)} = \{g \in \G \mid [g,h] \in \G_{(i)} \ \forall h \in \G\}$.

Let $g \in \G_{(1)}$, then for $y \in U$, we have 
$h \cdot g \cdot y = g \cdot h \cdot y = g \cdot y$. Thus $\Phi(h)$ acts as the identity on the union of the clopen sets $\{g \cdot U \mid g \in \G_{(1)} \}$.

For $1 \leq i < k$, assume that $\Phi(h)$ acts as the identity on the union of the clopen sets $\{g\cdot U \mid g \in \G_{(i)} \}$.
Let $g \in \G_{(i+1)}$ then for $y \in  U$, we have $[g,h] = g^{-1}h^{-1}g h \in \G_{(i)}$ and
$$h \cdot g \cdot y = h \cdot    g \cdot [g,h]  \cdot y = h \cdot g \cdot g^{-1}h^{-1}g h \cdot y =  g \cdot h \cdot y = g \cdot y.$$

Thus, $\Phi(h)$ acts as the identity on the union of the clopen sets $\{g\cdot U \mid g \in \G_{(i+1)} \}$.
It follows by finite induction that $\Phi(h)$ acts as the identity on   the union of the clopen sets $\{g \cdot U \mid g \in \G\}$,  which is all of $\fX$. Since the action is effective, $h = e$. Thus for every $g \in \G$ the set ${\rm Fix}(g)$ is nowhere dense, and $\Iso(\Phi)$ is meager.
  \endproof

We recall   another  property of groups that has   implications for studying odometer actions.
\begin{defn}[Unique root property] \label{def-uniqueroot}
 A group $\G$ satisfies the \emph{unique root property} if  for any $n\geq 2$ and $g,h \in \G$, 
$g^n = h^n$ implies $g= h$. 
\end{defn}
Free groups, torsion-free nilpotent   groups, and fundamental groups of closed orientable surfaces have the unique root property (see \cite[Lemma~2.2]{BartholdiBogopolski2010}).

This property is related to the notion of $\cD$-group, or a \emph{divisible} group, which assumes that in addition, every element of $\G$ has a unique $n^{th}$ root for all $n > 0$.
Section~2 of \cite{BartholdiBogopolski2010} gives a series of consequences of the unique root property. The proofs of the following are immediate.
\begin{lemma}
If $\G$ has the unique root property, then $\G$ is torsion-free.
\end{lemma}

\begin{lemma}\cite[Lemma~2.4]{BartholdiBogopolski2010}\label{lem-uinquerootext}
 Suppose $\G$ satisfies the unique root property, and $H  \subset \G$ has 
finite index, then the natural map $\Aut(H) \to  \Comm(\G)$ is injective.
\end{lemma}

\section{Lifting dynamical commensurators}\label{sec-lifting}

   In this section, we use the dynamical properties of the monodromy odometer      $(X_{\cG}, \G, \Phi)$, as  discussed in the previous section, to prove the existence of the map $\Theta_{\cG}$ in diagram \eqref{eq-diagram}. This yields the proof of  Theorem \ref{thm-injective}.  As a consequence, we obtain  conditions such that  a commensurator      $[\phi] \in \Comm(\G)$ is in the image of the characteristic map $\chi \colon \MCG(\fM_{\cP} ; \whx)  \to \Comm(\G)$. 
    
 We restate Theorem \ref{thm-injective} in Theorem \ref{thm-gottalift} for convenience of the reader. Given a basepoint $\whx \in \fM_{\cP}$ we denote by $x$ the corresponding point in the transversal space  $X_{\cG}$.
  
   \begin{thm}\label{thm-gottalift}
 Let    $(X_{\cG}, \G, \Phi)$  be  an effective and coherent   odometer.   Let  $x \in X_{\cG}$. Then there exists an injection 
$$\Theta_{\cG} \colon \Comm(X_{\cG}, \G, \Phi, x) \to \Comm(\G)$$
such that  $\chi = \Theta_{\cG} \circ \sigma_{\cP,  \whx}$. That is, the diagram \eqref{eq-diagram} commutes.  
 \end{thm}
  \proof

We first sketch the idea of the proof.
 For $h \in \Homeo(\fM_{\cP} ; \whx)$, the image $\chi(h) \in \Comm(\G)$ is defined using the zig-zag maps in \eqref{eq-shuffle} which exist  by the theorems of  Mioduszewski   \cite{Mioduszewski1963} and   Rogers and Tollefson  \cite{RT1972b}. 
More precisely,  
 $\chi(h)$ is defined using a map  $\vp_{i_{\ell}} \colon M_{{i_{\ell}}} \to M_{{\ell}}$ between two  stages of the presentation $\cP_{\cG}$,  which induces the commensurator map as in \eqref{eq-defchi}.

 The map  $\sigma_{\cP,  \whx} \colon \MCG(\fM_{\cP} ; \whx) \to \Comm(X_{\cG}, \G, \Phi, x)$  in  Theorem~\ref{thm-charmap}  is defined   by restricting a given homeomorphism $h \in \Homeo(\fM_{\cP} ; \whx)$ to a transversal neighborhood through $\whx$, yielding a return equivalence    $h_U \colon U \to V$. Then $h_U$ induces a map $\Theta_U \colon \cH_U \to \cH_V$ of holonomy groups.

 Thus, the map $\Theta_{\cG}$   is obtained by lifting the isomorphism $\Theta_U \colon \cH_U \to \cH_V$, which is the   holonomy action on a local transversal  to the fibration $\fM_{\cP} \to M_0$, to a virtual map   between the fundamental groups of the manifolds in the tower $\cP_{\cG}$.  
We first prove a technical result.  
 
   \begin{lemma}\label{lem-lifting}
  Let  $(X_{\cG}, \G, \Phi)$  be  an effective and coherent odometer, and let $U \subset X_{\cG}$ be an adapted set. Then there exists a  finite index subgroup $\G_U' \subset \G_U$   such that the     map  $\Phi_U \colon \G_U' \to \Homeo(U)$ is   injective.
  \end{lemma}
  \proof
 By assumption, the action is effective, so $\Phi \colon \G \to \Homeo(X_{\cG})$ is injective. Moreover,   the isotropy subgroup  $\G_U \subset \G$ has finite index.
 The kernel $K_U = \ker \{ \Phi_U \colon \G_U \to \Homeo(U)\} \subset \G$ is a finite group by the coherent hypothesis.
 If the   action is  topologically free, then in fact $K_U$ is trivial and so we can take $\G_U' = \G_U$.

Otherwise, the finite group $K_U$ acts effectively on $X_{\cG}$, so for a sufficiently small adapted set $U' \subset U$, every element of  the group $K_U$ acts non-trivially on the set of translates of $U'$. That is,  the   map  $\Phi_U \colon K_U \to \Homeo(\G_U/\G_{U'})$ is injective.

 Let $C_{U'} = \ker \{ \Phi_{U'} \colon \G \to \Homeo(\G/\G_{U'}) \}$ be the kernel of the restricted action, which is a normal subgroup with finite index in $\G$. Then   $g \in C_{U'}$ fixes all points in $\G/\G_{U'}$, while $K_U$ acts effectively on $\G/\G_U$, so $g \not\in K_U$. That is, we have  $K_U \cap C_{U'} = \{e\}$, where $e$ is the identity in $\G$.

Then $\G'_U =  \G_U \cap C_{U'}$ has finite index in  $\G$, and   the restriction $\Phi_U \colon \G'_U \to \Homeo(U)$ is injective.
\endproof

We now construct the map $\Theta_{\cG}$.
 First, by Definition~\ref{def-dyncomm2}, an element  of $\Comm(X_{\cG}, \G, \Phi, x)$  is represented  by a map $h_U \colon U \to V$, where  $U, V \subset X_{\cG}$ are adapted sets with $x \in U \cap V$  and $h_U(x) = x$, and  such that the induced map $\Theta_U$  is an isomorphism of     $\cH_U  =\Phi_U(\G_U) \subset \Homeo(U)$ with  $\cH_{V} =\Phi_{V}(\G_V)  \subset \Homeo(V)$.
  
 By    Lemma~\ref{lem-lifting}, there exists   $\G'_{U} \subset \G_U$ with finite index such that    $\Phi_U \colon \G'_U \to \Homeo(U)$ is injective.

Similarly, the map $\Phi_V \colon \G'_V \to \Homeo(V)$ has finite kernel, and  there exists   $\G'_{V} \subset \G_V$ with finite index such that    $\Phi_V \colon \G'_V \to \Homeo(V)$ is injective.

Now form the ```intersection pull-back'',  which has finite index in $\G$, 
\begin{equation} \label{eq-induced}
\G''_U = \Phi_U^{-1} \left\{ \Phi_U(\G'_U) \cap \Theta_U^{-1}(\Phi_V(\G'_V)) \right\} \ .
\end{equation}
 
The composition $\phi = \Phi_V^{-1} \circ \Theta_U \circ \Phi_U \colon \G''_U \to \G''_V$ defines  a commensurator   for $\G$, where $\G_V'' = \phi(\G_U'')$.  
Then define  $\Theta_{\cG}[h_U,U,V]  = [\phi] \in \Comm(\G)$.

Now suppose there are given two representatives of $[h_U,U,V]$, say $h_U \colon U \to V$ and $h'_{U'} \colon U' \to V'$, such that $(h_U,U,V) \comm (h'_{U'},U',V')$. Then there exists an adapted set $U''$ with $x \in U'' \subset U \cap U'$ such that $h_U|U'' = h'_{U'}|U''$.

Let $\phi: \G''_U \to \G''_V$ and $\phi': \G''_{U'} \to \G''_{V'}$ be the commensurators defined by \eqref{eq-induced}.
The stabilizer $\G_{U''}$ has finite index in both $\G_U$ and $\G_{U'}$ so the restrictions of the maps $\phi $ and  $\phi'  $ agree on the intersection $\G_{U''} \cap \G_U''  \cap \G_{U'}''$. Thus, the maps $\phi$ and $\phi'$ determine the same class in $\Comm(\G)$, so $\Theta_{\cG}[h_U,U,V]$ is well-defined.

Finally, for $h \in \Homeo(\fM_{\cP} ; \whx)$, we   show  that the diagram \eqref{eq-diagram} commutes for the map $\Theta_{\cG}$. That is, we have  $\chi(h) = \Theta_{\cG} \circ \sigma_{\cP,  \whx}(h)$.     
Recall that  the map $\chi$ is constructed from an approximation  $\{\varphi_{i_\ell} \colon M_{i_{\ell}} \to M_i \mid \ell \geq 0\}$    to $h$, and then the initial map $\vp_{i_{0}} \colon M_{i_{0}} \to M_{0}$  
 induces the inclusion $\vp_* \colon \G_{i_0} \to \G_0$.  The commensurator $\chi(h) \in \Comm(\G)$ is then defined as in diagram \eqref{eq-defchi}.

The map  $\vp_* \colon \G_{i_0} \to \G_0$ restricts to inclusions  on the  subgroups $\G_{i_{\ell}} \subset \G_{i_0}$ for $\ell > 0$, so  induces a map $\whvp_{i_{0}} \colon U_{i_{\ell}} \to X_{\cG}$ by the formula \eqref{eq-invcosetlim}.  The action of $\G_{i_{\ell}}$ on $U_{i_{\ell}}$ is given by restricting the left action of $\G$  on $X_{\cG}$. 

On the other hand,  the map  $\sigma_{\cP,  \whx}$ is defined by   restricting the homeomorphism $h$ to a transversal $x \in U \subset X_{\cG}$ through the basepoint $\whx$. By modifying $h | U$ by a small leafwise homotopy, we may assume that $h$ maps $U$ into the fiber $X_{\cG}$. We apply this construction to the clopen subset $U = U_{i_{\ell}}$ for $\ell$ sufficiently large, to obtain 
$h| U_{i_{\ell}} = \whvp_{i_{\ell}} \colon U_{i_{\ell}} \to X_{\cG}$.
 Then      $\Theta_U \colon \cH_{U_{i_{\ell}}} \to \cH_V$   is induced from  the map $\vp_* \colon \G_{i_0} \to \G_0$.
This induced map is by definition the class   $ \sigma_{\cP,  \whx}(h) = [h_U,U,V]$. 

Then $ \Theta_{\cG} \circ \sigma_{\cP,  \whx}(h) = \Theta_{\cG} [h_U,U,V] $ is the commensurator  $\G_U'' \to \G_V''$ constructed as in \eqref{eq-induced} from a lift of the map  $\Theta_U \colon \cH_{U_{i_{\ell}}} \to \cH_V$. But this lift is just the restriction of 
$\vp_* \colon \G_{i_0} \to \G_0$ to $\G_U'' \subset \G_{i_0}$, which equals the commensurator $\chi(h)$ defined above. 
  \endproof

\section{The Borel Property}\label{sec-classify}

Recall that a finite $CW$-complex  $Y$ is \emph{aspherical} \cite{FarrellJones1990} if it is connected  and its universal covering space  is  contractible. 
If $Y$ and $Y'$ are aspherical, then an isomorphism $\alpha \colon \pi_1(Y, y_0) \to \pi_1(Y', y_0')$ induces a homotopy equivalence $h_{\alpha} \colon Y \to Y'$ with $h_{\alpha}(y_0) = y_0'$. There are many classes of closed manifolds which are aspherical, including all surfaces of genus $g \geq 1$, and many irreducible 3-manifolds. Chapter 2 of \cite{Weinberger2023} gives an overview of constructions of many further classes of examples.

Based on the examples of surfaces and Mostow Rigidity for negatively curved manifolds, one has:
\begin{defn}\label{def-toprigid}
A closed manifold $N$ is said to be \emph{topologically rigid} if, given a closed manifold $M$, every homotopy equivalence $h \colon M \to N$   is homotopic to a homeomorphism.
\end{defn}
Following remarks by Borel, the following became the inspiration for many remarkable results:
\begin{conj}[Borel Conjecture]
Every aspherical closed manifold is topologically rigid.
\end{conj}
 
\begin{remark}\label{rmk-borelex}
The Borel Conjecture has been proven for many classes of aspherical manifolds:
\begin{itemize}
\item the torus $\mT^n$ for all $n \geq 1$,
\item all  infra-nilmanifolds,
\item   closed Riemannian manifolds $Y$ with negative sectional curvatures,
\item   closed Riemannian manifolds $Y$ of dimension $n \ne 3,4$ with non-positive sectional curvatures,
\item all closed irreducible 3-manifolds.
\end{itemize}
\end{remark}
Chapter 2 of  \cite{Weinberger2023} gives a more extensive description of manifolds which satisfy the Borel Conjecture.
See also   \cite{DavisHillman2025}, and     \cite{Luck2026} for further  examples.
Each class of manifolds in Remark~\ref{rmk-borelex} is closed under taking covers,  which was the motivation for introducing the terminology   ``strongly Borel''   in the work \cite{CHL2019a}, and also the notion of ``ultra Borel'' as follows:
\begin{defn}\label{def-stronglyborel}
A closed aspherical manifold $M$ is said to be \emph{Borel} if it satisfies the Borel Conjecture. It is said to be \emph{strongly Borel} if every finite covering of $M$ satisfies the Borel Conjecture, and is said to be \emph{ultra Borel} if, in addition, the induced homeomorphism is unique up to isotopy.
\end{defn}
The uniqueness property is discussed by Weinberger in terms of the surgery sets for aspherical manifolds in \cite[Chapter 3]{Weinberger2023}, and by Davis and Hillman in \cite[Section 3]{DavisHillman2025}. All of the classes of manifolds above satisfy the ultra Borel property.
  
  We next apply    the Borel Property to the solution of   Problem~\ref{prob-chi}. Let $M_0$ be a closed aspherical manifold with $\G =\pi_1(M_0, x_0)$, and let $\phi \colon H \to K$ be a commensurator for $\G$. Let $M_H \to M_0$ be the covering determined by $H$, and $M_K \to M_0$ the covering determined by $K$. Then $\phi$ induces a homotopy equivalence $g_{\phi} \colon M_H \to M_K$. Suppose that $M_0$ is ultra Borel, then $g_{\phi}$ determines a unique isotopy class of homeomorphisms $\whg_{\phi} \colon M_H \to M_K$.
 
Given  a solenoidal manifold $\fM_{\cP}$ over $M_0$, and  $\phi \colon H \to K$, we next consider when does the induced map on coverings $\whg_{\phi} \colon M_H \to M_K$ induce a map $h_{\phi} \in \Homeo(\fM_{\cP};  \whx)$ with $\chi(h_{\phi}) = [\phi]$?
The answer to this uses  the   suspension construction from the theory of foliations, which  is discussed   for solenoidal manifolds in \cite[Theorem~3.2]{CHL2019a}, and more generally for foliated spaces  in \cite[Section~3.1]{CandelConlon2000}.

  Let $(X_{\cG}, \G, \Phi)$ denote the odometer associated to a group chain $\cG$ in $\G$ as in \eqref{eq-invcosetlim}. Let $\fM_{\cP}$ denote the solenoidal manifold associated to the presentation $\cP_{\cG}$ as in \eqref{eq-invlimit}. For a point $\whx \in \fM_{\cP}$, denote by $x = (x_\ell) \in X_{\cG}$ the corresponding point in the Cantor space of the odometer model.
  For each $\ell \geq 0$, the coordinate projection  $\whq_{\ell} \colon \fM_{\cP} \to M_{\ell}$  is a fibration, where the fiber   $\fX_{\ell} = \whq_{\ell}^{-1}(x_{\ell})$ is  a Cantor space, and is identified with the closure $U_{\ell} \subset X_{\cG}$ of the orbit of $\G_{\ell} \subset \G$ acting on the basepoint $x$. Then the restriction of $\Phi$ to $U_\ell$ yields an odometer action $\Phi_{\ell} \colon \G_{\ell} \times U_{\ell} \to U_{\ell}$.

Let $\wtM$ denote the universal covering of  $M$ with the right action of $\G$ by deck translations. We convert this to a left action in the usual way, setting $\gamma \cdot x = x \cdot \gamma^{-1}$. Then we obtain the left action of $\G$ on $\wtM \times U_{\ell} \to U_{\ell}$, where $\gamma \cdot (z,y) = (z \cdot \gamma^{-1} , \Phi_{\ell}(\gamma)(y))$ for $\gamma \in \G_{\ell}$.
Then define  
\begin{equation}\label{eq-suspensionfols}
\fM_{U_{\ell}} = \wtM \times U_{\ell} / (z , y) \sim (z \cdot \gamma^{-1} , \Phi_{\ell}(\gamma)(y))  ~ {\rm for }~ z \in \wtM , ~ \gamma \in \G_{\ell}, ~ y \in U_{\ell},
\end{equation}
which is a   foliated space   with dense leaves \cite{CandelConlon2000,MS2006}.   
\begin{thm}   \label{thm-weaksuspensions}
For each $\ell \geq 0$, there is a canonical homeomorphism $\tau_{\ell}  \colon \fM_{U_{\ell}} \to \fM_{\cP}$.
\end{thm}
\proof
For   all $\ell \geq 0$,   there is a natural homeomorphism 
$\rho_{\ell} \colon  \fM_{U_{\ell}} \to \fM_{\cP_{\ell}}$, where   $\fM_{\cP_{\ell}}$ is the  solenoidal manifold associated to the truncated presentation  as defined in \eqref{eq-presentationell}. This is  a consequence of the lifting property for maps between coverings, as shown in \cite{ClarkHurder2013}.  By Lemma~\ref{lem-canonical},  there is a canonical homeomorphism $\sigma_{\ell} \colon \fM_{\cP_{\ell}} \cong  \fM_{\cP}$. Then set $\tau_{\ell} =  \sigma_{\ell} \circ \rho_{\ell}$. 
\endproof

We single out the special case $\ell =0$ in the above suspension construction:
\begin{defn}\label{def-suspension}
The quotient   $\fM_{\Phi} = \G \backslash (\wtM_0 \times X_{\cG})$ is called the \emph{suspension model} for $\fM_{\cP}$.
\end{defn}

 The basepoints  $x = \bigcap_{\ell > 0} U_{\ell} \in X_{\cG}$ and $x_{\ell} \in M_{\ell}$ define  a basepoint   $\whx_{\ell} \in \fM_{U_{\ell}}$, which is the image of $(\wtx_0, x) \in \wtM \times U_{\ell}$ under the quotient  in \eqref{eq-suspensionfols}.
For $\wtx = \wtx_0$, the leaf $L_{\wtx} \subset \fM_{\cP}$ of $\F_{\cP}$ containing $\wtx$ is the image of $\wtM \times \{x\}$
 under the quotient \eqref{eq-suspensionfols}, and so the fundamental group 
 $$\pi_1(L_{\wtx}, \wtx) = \Iso(\{x\})  = \{\gamma \in \G \mid \gamma \cdot x = x\} \  .$$
Combine this observation with  Lemma~\ref{lem-topfree} to obtain:

 \begin{lemma}\label{lem-simplyconnected}
 Let $\fM_{\Phi}$ be     the suspension of  $(X_\cG,\G,\Phi)$. Suppose that $L_{\whx}$ is simply connected, then 
  $\ker (\cG) =  \bigcap_{\ell \geq 0} \G_{\ell} = \{e\}$, and hence the action  $(X_\cG,\G,\Phi)$ is topologically free.
 \end{lemma}
 
 We now     consider  the map  $\ds 
 \sigma_{\Phi,  \whx} \colon \MCG(\fM_{\cP} ; \whx) \to \Comm(\fX_0, \G, \Phi, x)$, and the proof of Theorem~\ref{thm-surjective}. 
We are given that $M_0$ is an ultra Borel   manifold,  and $\fM_{\cP} \cong \fM_{\Phi}$ is a solenoidal manifold over $M_0$ with  a simply connected leaf $L_{\wtx}$. 

   Let $[h_U, U, V] \in \Comm(\fX_0, \G, \Phi, x)$ be represented by the local return map $h_U \colon U \to V$. We first show there exists $h \in \Homeo(\fM_{\cP}, ; \whx)$ such that  $ \sigma_{\cP,  \whx}(h) = [h_U, U, V]$.  By Lemma~\ref{lem-simplyconnected} the odometer   $(X_{\cG}, \G, \Phi)$  is  topologically free, so by  Theorem~\ref{thm-gottalift}  there exists an injection 
$$\Theta_{\cG} \colon \Comm(X_{\cG}, \G, \Phi, x) \to \Comm(\G)  \  .$$

That is, the local return map $h_U$ can be lifted to a commensurator $\phi_U$ for $\G$. Recall its construction.
Let $\ell > 0$ be sufficiently large so that $x \in U_{\ell} \subset U$.
Let $V_{\ell} = h_U(U_{\ell})$ and note that this is an adapted  neighborhood of $x$ (see   \cite[Lemma~3.4]{CHL2019a} for a proof). 
Then $h_U$ induces an isomorphism $\Theta_{U_{\ell}} \colon \cH_{U_{\ell}} \to \cH_{V_{\ell}}$.
 The action $(X_\cG,\G,\Phi)$ is topologically free, so the restricted action maps 
 $$\Phi_{U_{\ell}} \colon \G_{U_{\ell}} \to \Homeo(U_{\ell}) \quad , \quad \Phi_{V_{\ell}} \colon \G_{V_{\ell}} \to \Homeo(V_{\ell})$$
 are both injective. Thus,  
  $\Theta_{U_{\ell}}$ lifts to an isomorphism  $\phi_{\ell} \colon \G_{U_{\ell}} \to \G_{V_{\ell}}$.
  Then $\Theta_{\cG}([h_U, U, V]  = [\phi_{\ell}]$.

 Let $M_{U_{\ell}} \to M$ be the covering associated to $\G_{U_{\ell}}$ and $M_{V_{\ell}} \to M$  the covering associated to $\G_{V_{\ell}}$. Then   $\phi_{\ell}$ induces a homotopy equivalence $\ophi_{\ell} \colon M_{U_{\ell}} \to M_{V_{\ell}}$, and by assumption that $M$ is ultra Borel, the map $\ophi_{\ell}$ can be chosen to be a homeomorphism, and the isotopy class of the homeomorphism is uniquely defined by $\phi_{\ell}$.  
 
 We have $x \in V_{\ell}$ so there exists $\ell' > \ell$ such that $x \in U_{\ell'} \subset V_{\ell}$, and so $\G_{\ell'} \subset \G_{V_{\ell}} \subset \G$. Define
$$\fM_{V_{\ell}} = \wtM \times V_{\ell} / (z , y) \sim (z \cdot \gamma^{-1} , \Phi_{\ell}(\gamma)(y))  ~ {\rm for }~ z \in \wtM , ~ \gamma \in \G_{V_{\ell}}, ~ y \in V_{\ell} \ . $$
The suspension of the map $\ophi_{\ell}$ yields a homeomorphism of fibrations 
\begin{equation}
\whphi_{\ell} \colon \fM_{U_{\ell}} ~ \cong  \G_{U_{\ell}}\backslash (\wtM \times U_{\ell})   \quad  \stackrel{\phi_{\ell} \times \Theta_{U_{\ell}}}{\longrightarrow} \  \G_{V_{\ell}} \backslash (\wtM \times V_{\ell})  ~ \cong  \fM_{V_{\ell}} \ .
\end{equation}

\bigskip

Then  there is a homeomorphism $\tau_{\ell}^{V} \colon \fM_{V_{\ell}} \to \fM_{\cP}$ of fibrations over $M_0$, which follows by the same proof as for   Theorem~\ref{thm-weaksuspensions}.
  Now   define 
  $h_{\phi} \colon \fM_{\cP} \to \fM_{\cP}$ as the composition of homeomorphisms
\begin{equation}\label{eq-twisted}
h_{\phi} \colon \fM_{\cP} \stackrel{\tau_{\ell}^U}{\longleftarrow} \fM_{U_{\ell}} \stackrel{\whphi_{\ell}}{\longrightarrow}   \fM_{V_{\ell}} \stackrel{\tau_{\ell}^{V}}{\longrightarrow} \fM_{\cP}  \  .
\end{equation}

Thus, we have     $h_{\phi} \in \Homeo(\fM_{\cP};  \whx)$ with  $\chi(h_{\phi}) = [\phi]$. This completes the proof of Theorem~\ref{thm-surjective}.

The following summary of the results   above   is used in the proofs of   Corollaries~\ref{cor-homogeneous} and  \ref{cor-nonhomogeneous}, and  used in the calculations of examples in Section~\ref{sec-examples}.
  \begin{thm} \label{thm-summary}
Let  $M_0$ be an ultra  Borel   manifold, and  $\fM_{\cP}$ be a solenoidal manifold over $M_0$ defined by a group chain $\cG$.  Assume that the leaf $L_{\whx}$ through the basepoint $\whx$  is simply connected.  
Then  
$\ds 
 \sigma_{\cP,  \whx} \colon \MCG(\fM_{\cP} ; \whx) \to \Comm(\fX, \G, \Phi, x)  
$
is an isomorphism, and  $\Theta_{\cG} \colon  \Comm(\fX, \G, \Phi, x)   \to \Comm(\G)$ is injective.
 
  \end{thm}

\section{Normalizers} \label{sec-normalizer}
 
In this section, we analyze the group   $\MCG(\fM_{\cP})$ as a topological extension of  $\MCG(\fM_{\cP} ; \whx)$.
We use the isomorphism  $\fM_{\Phi} \cong \fM_{\cP}$, where $\fM_{\Phi}$ is   the suspension model in Definition~\ref{def-suspension} associated to  the odometer action $(X_{\cG}, \G, \Phi)$.  
Then the normalizer group $\fN_{\cG}$   acts via right translations   on  $\fM_{\Phi}$. 
The problem is then to understand how the action of $\fN_{\cG}$ interacts with the elements of $\Homeo(\fM_{\Phi} ; \whx)$, which can be subtle. 
 The   interaction of  commensurators in $\Comm(\G)$   with translations  by elements of $\fN_{\cG}$    plays a central role in the calculation of  $\MCG(\fM_{\Phi})$ for non-homogeneous solenoids.

Recall that $\fG_{\cG}$ is the profinite completion of $\G$ with respect to the group chain $\cG$, and there is an induced isometric action $\oPhi \colon \fG_{\cG} \times X_{\cG} \to X_{\cG}$.  The  basepoint  $x \in X_{\cG}$ is defined by the chain $\cG$, and    $\fD_{\cG} = \fD_{\cG, x} = \{ \whg \in \fG_{\cG} \mid \oPhi(\whg)(x) = x\}$ is the isotopy subgroup at  $x$.  Then  there is an isomorphism $X_{\cG} \cong \fG_{\cG}/\fD_{\cG}$ which is  equivariant for the left  $\fG_{\cG}$-action.
 
  Let  $\fN_{\cG} = \fN_{\fG_{\cG}}(\fD_{\cG})$ denote  the normalizer of   $\fD_{\cG} \subset \fG_{\cG}$, and     $\fA_{\cG} = \fD_{\cG} \backslash \fN_{\cG}$ be the quotient group.

We begin our analysis of  $\MCG(\fM_{\cP})$ by first considering the related problem  for an odometer action $(X_{\cG}, \G, \Phi)$. 
Recall   the isomorphism  group $\fI(X_{\cG}, \G, \Phi)$ of the  action as given by  Definition~\ref{def-normalizer}.
\begin{lemma}\label{lem-quotient}
The right action $X_{\cG} \times \fN_{\cG} \to X_{\cG}$ descends to an action $X_{\cG} \times \fA_{\cG} \to X_{\cG}$, and yields    an inclusion $\rho_{\cG} \colon \fA_{\cG} \to \fI(X_{\cG}, \G, \Phi)$.
\end{lemma}
 \proof
 The right action of $\fN_{\cG}$ on $X_{\cG} = \fG_{\cG}/\fD_{\cG}$ commutes with the left action $\Phi$ of $\G$ on $X_{\cG}$, and descends to a right action by the quotient group $\fA_{\cG} = \fD_{\cG} \backslash \fN_{\cG}$.
  \endproof

 Let $\Theta_{\cG} \colon \fI(X_{\cG}, \G, \Phi) \to \Aut(\G)$ be the ``time change'' group automorphism in Definition~\ref{def-normalizer}, and let $\Aut_{\cG}(\G)$ denote the image of $\Theta_{\cG}$.  In the context of $\mR^n$-tiling space isomorphisms, Kwapisz calls this the \emph{General Linear Group} of the action \cite{Kwapisz2011}, as it has   range in ${\rm GL}(\mR^n)$.   Introduce
  \begin{equation}\label{eq-normalizerfp}
\fI(X_{\cG}, \G, \Phi, x) = \{ h \in \fI(X_{\cG}, \G, \Phi) \mid h(x) = x \}  \  .
\end{equation}

\begin{lemma}\label{lem-cG-iso}
The restriction $\Theta_{\cG} \colon \fI(X_{\cG}, \G, \Phi, x) \to \Aut_{\cG}(\G)$ is an isomorphism.
\end{lemma}
\proof
Let $h \colon X_{\cG} \to X_{\cG}$ be an isomorphism such that $h(x) = x$, with   $\Theta = \Theta_{\cG}(h) \colon \G \to \G$ the identity. 
The action $\Phi \colon \G \times X_{\cG} \to X_{\cG}$ is minimal, so the orbit of $x$ is dense in $X_{\cG}$.
By the condition \eqref{eq-isomorphism}, we then have
$h \circ \Phi(g) =   \Phi(g) \circ h$ for all $g \in \G$. Apply this to $x$ to obtain $h(\Phi(g) (x)) = \Phi(g)(x)$, 
and so $h$ is the identity on a dense subset of $X_{\cG}$, which yields the claim.
\endproof

The classical results on the Galois automorphisms of a  finite  covering space (as given, for example, in   \cite[Corollary~7.3]{Massey1967}) can be extended  to  the normalizers of an odometer action, as given  in \cite[Section~2]{DHL2016a}.   The following is the key result, adapted  to the case of an odometer action.

\begin{lemma}\label{lem-normalizer}
Let $h \in \fI(X_{\cG}, \G, \Phi)$    be in the kernel of $\Theta_\cG$. Let $y = h(x)$ and suppose that $\whg \in \fG_{\cG}$  is such that $\oPhi(\whg)(x) = y$.  Then   $\whg \in \fN_{\cG}$.
\end{lemma}
\proof

We show that  for all $\whk \in \fD_{\cG}$, the conjugate   $\whg^{-1}   \whk   \whg \in \fD_{\cG}$, and thus  $\whg \in \fN_{\cG}$.

  The map $h$ is an isomorphism of the action, and we assume that $\Theta = \Theta_\cG(h)$ is the identity, so by Definition~\ref{def-isomorphism}  we have  $ h \circ \Phi(g) =  \Phi(g) \circ h$ for all $g \in \G$.

Let  $\oG_{\ell}$ denote the closure of $\G_{\ell}$ in $\fG_{\cG}$.  Then $\oG_{\ell}$  acts transitively on $U_{\ell}$, hence $U_{\ell} = \oG_{\ell}/\fD_{\cG}$ for all $\ell$. It follows that  $\bigcap_{\ell} \oG_{\ell} = \fD_{\cG}$. 
Thus, given  $\ok \in \fD_{\cG}$,  there exists a sequence $\{k_{\ell} \in \G_{\ell} \mid \ell > 0\}$ such that $\ds \lim_{\ell \to \infty} \ok_{\ell} = \ok$, where $\ok_{\ell}  \in \oG_{\ell}$ is the image of $k_{\ell}$ in $\oG_{\ell}$.

 We have that $\oPhi(\ok)(x) = x$ and $h(x) = y$, so 
 $\ds \lim_{\ell \to \infty} \Phi(k_{\ell}) (x) = x$ which implies 
    $\ds \lim_{\ell \to \infty}  \Phi(k_{\ell})(y) = y$, so 
$$
  \oPhi(\whg)(x) = y = \lim_{\ell \to \infty} \Phi(k_{\ell}) ( y) = \lim_{\ell \to \infty}\Phi(k_{\ell}) (\oPhi(\whg)(x))  = \lim_{\ell \to \infty}\Phi(\ok_{\ell})\oPhi( \whg)(x) \ ,
$$
and therefore $\ds x =  \lim_{\ell \to \infty}\oPhi(\whg^{-1} \ok_{\ell}   \whg)(x) = \oPhi(\whg^{-1} \ok   \whg)(x)$. 
Thus $\whg^{-1} \ok   \whg \in \fD_{\cG}$ for all $\ok \in \fD_{\cG}$,  and so $\whg \in \fN_{\cG}$ as was to be shown.
\endproof

We combine these     results to obtain:
 \begin{prop}\label{prop-product}
There is an exact sequence 
\begin{equation}\label{eq-isoext}
\{1\} \to  \fA_{\cG} \stackrel{\rho_{\cG}}{\longrightarrow}   \fI(X_{\cG}, \G, \Phi)  \stackrel{\Theta_{\cG}}{\longrightarrow} \Aut_{\cG}(\G) \to \{1\}
\end{equation}
 \end{prop}

  \proof

Let $h \in \fI(X_{\cG}, \G, \Phi)$  which is in the kernel of $\Theta_{\cG}$. That is, $h$ is an automorphism of the action. Let $y = h(x)$ then by Lemma~\ref{lem-normalizer} there exists $\whg \in \fN_{\cG}$ such that $\oPhi(\whg)(x) = y$, that is, the coset $y = \whg \cdot \fD_{\cG} = \fD_{\cG} \cdot \whg$ as $\whg$ normalizes $\fD_{\cG}$. 
 Therefore $\rho_{\cG}(\whg^{-1}) \circ h (x) = x$. 
 Then by  Lemma~\ref{lem-cG-iso} the composition $\rho_{\cG}(\whg^{-1}) \circ h$ is the identity, hence $h = \rho_{\cG}(\whg)$, so that $h$ is in the image of $\rho_{\cG}$ in $\Homeo(X_{\cG})$. 
 \endproof
 
The proof of Proposition~\ref{prop-product}   depends essentially on Lemmas~\ref{lem-cG-iso} and \ref{lem-normalizer}. This approach does not work for $\MCG(\fM_{\cP})$, as the subgroup $\Homeo_0(\fM_{\cP})$ does not preserve the basepoint $\whx$, so there are no results corresponding to these lemmas for $\MCG(\fM_{\cP})$.
However, $\Homeo_0(\fM_{\cP})$ does preserve the    baseleaf $L_{\whx}$ through $\whx$.    Odden   used this observation  in \cite{Odden2005} to give a   more geometric approach   to constructing  $\MCG(\fM_{\cP})$ from  $\Homeo(\fM_{\cP} ; \whx)$ for coverings of surfaces, and we generalize his approach to all solenoids.  We first recall    the suspension model $\fM_{\Phi}$ for $\fM_{\cP}$. 
   
Define    a  left action of $\G$ on $\wtM_0 \times X_{\cG}$   by 
\begin{equation}\label{eq-coveringaction}
\gamma \cdot (z,y) = (z \cdot \gamma^{-1}, \Phi(\gamma)(y)) \quad {\rm for} ~ z \in \wtM_0 ,  ~ y \in X_{\cG} ~ , ~ \gamma \in \G \ .
\end{equation}
Then  $\fM_{\Phi} = \G \backslash (\wtM_0 \times X_{\cG})$, and by Theorem~\ref{thm-weaksuspensions} there is a natural homeomorphism $\tau  \colon \fM_{\Phi} \to \fM_{\cP}$.

Let   $q_{\Phi}  \colon \fM_{\Phi} \to M_0$   denote the projection onto the first factor composed with the universal covering projection.
Then identify $X_{\cG} =   q_{\Phi}^{-1}(x_0)$ where $x_0 \in M_0$ is the basepoint.
 The basepoint  $\whx \in \fM_{\Phi}$ is the image of   $(\wtx_0, x)$, where $\wtx_0 \in \wtM_0$ is a lift of the basepoint $x_0 \in M_0$ and $x \in X_{\cG}$ is the basepoint defined by the group chain $\cG$.  
With these choices, the  homeomorphism $\tau  \colon \fM_{\Phi} \to \fM_{\cP}$ identifies $X_{\cG} \subset \fM_{\Phi}$ with the fiber $\fX_0 = \whq_0^{-1}(x_0) \subset \fM_{\cP}$.   We work with the suspension model $\fM_{\Phi}$ and identify $X_{\cG}$ with a transversal in $\fM_{\Phi}$.

 Introduce the subgroup of \emph{baseleaf preserving} homeomorphisms:
 \begin{equation}\label{eq-baseleafhomeos}
\Homeo(\fM_{\Phi} ; L_{\whx}) =  \{ h \in \Homeo(\fM_{\Phi}) \mid h(L_{\whx}) = L_{\whx}\} \ .
\end{equation}
A homeomorphism preserves path connected components, so  $\Homeo(\fM_{\Phi};  \whx) \subset \Homeo(\fM_{\Phi} ; L_{\whx})$. Also note that $\Homeo_0(\fM_{\Phi}) \subset \Homeo(\fM_{\Phi} ; L_{\whx})$, as $h \in \Homeo_0(\fM_{\Phi})$ maps the basepoint $\whx$ to a point in the leaf $L_{\whx}$, hence preserves the leaf. 
Define:
 \begin{equation}\label{eq-baseleafmcg}
\MCG_{BLP}(\fM_{\Phi}) = \Homeo(\fM_{\Phi} ; L_{\whx})/\Homeo_0(\fM_{\Phi}) \subset \MCG(\fM_{\Phi}) \ . 
\end{equation}

 The following result was proved   for the universal hyperbolic solenoid by Odden \cite[Theorem~4.12]{Odden2005}. We give a proof  for all solenoidal manifolds.
  
\begin{prop}  \label{prop-baseleaf}
Assume  that  the baseleaf $L_{\whx} \subset \fM_{\Phi}$ is simply connected, then the inclusion $\Homeo(\fM_{\Phi} ; {\whx}) \subset \Homeo(\fM_{\Phi} ; L_{\whx})$ induces an   isomorphism $\MCG(\fM_{\Phi}; \whx) \cong  \MCG_{BLP}(\fM_{\Phi})$. 
\end{prop}
\proof
 
Let $h \in \Homeo(\fM_{\Phi} ; L_{\whx})$ and set $\why = h(\whx) \in L_{\whx}$. Choose an isotopy $\psi_t \in \Homeo_0(\fM_{\Phi})$ with $\psi_0 = id$, and $\psi_1(\why) = \whx$.   Then the composition $\psi_1 \circ h \in \Homeo(\fM_{\Phi} ; \whx)$.  Thus,   $\Homeo(\fM_{\Phi} ; L_{\whx})$ is generated by the subgroups $\Homeo(\fM_{\Phi}; \whx)$ and $\Homeo_0(\fM_{\Phi})$.  It follows that the inclusion composed with  the quotient map, denoted by 
   $\beta \colon \Homeo(\fM_{\Phi} ; \whx) \to \MCG_{BLP}(\fM_{\Phi})$,  is surjective.

Now suppose that $h \in \Homeo(\fM_{\Phi}; \whx)$ is in the kernel of $\beta$, so that $h$ isotopic to the identity map, say by a family $h_t \in \Homeo(\fM_{\Phi})$ with $h_0 = id$ and $h_1 = h$. This defines the map $H \colon \fM_{\Phi} \times [0,1] \to \fM_{\Phi}$ which lifts to a map
\begin{equation}
\wtH \colon \wtM_0 \times X_{\cG} \times [0,1] \to \wtM_0 \times X_{\cG} \quad , \quad  \wtH(\wtx_0, x, 0) =  (\wtx_0, x) \ . 
\end{equation}
 The homeomorphisms $h_t$ are leaf preserving, so the path  $\alpha(t) = h_t(\whx)$ satisfies $\alpha(0) = \whx$ and $\alpha(1) = h(\whx) = \whx$. As $L_{\whx}$ is simply connected, the path $\alpha$ is contractible relative the basepoint $\whx$, and so the holonomy transport of the transversal $X_{\cG}$ it determines is trivial. Thus, we can make an equivariant deformation of the lifted map $\wtH$ so that the path $\wtH(\wtx_0, x, t)$ is constant. Then  $h$ is isotopic to the identity map relative the basepoint $\whx$. It follows that $\beta$ descends to an isomorphism $\obeta \colon \MCG(\fM_{\Phi}; \whx) \to \MCG_{BLP}(\fM_{\Phi})$.
 \endproof

 Next, we define a right action $\rho_{\cG} \colon \fM_{\Phi} \times \fN_{\cG} \to \fM_{\Phi}$ of the normalizer $\fN_{\cG}$.   
 Let $\fN_{\cG}$ act      on   the product  $\wtM_0 \times X_{\cG}$ via the right action on the second factor.   
  This right action   commutes with the left action $\Phi$ of $\G$, so   descends to the action   $\rho_{\cG}$.   Note that this also yields  an induced action 
$\rho_{\cG} \colon \fM_{\Phi} \times \fA_{\cG} \to \fM_{\Phi}$, as  $\fD_{\cG}$ acts trivially on the factor $X_{\cG}$.  The restriction of the  right action $\rho_{\cG}$ to the  fiber $X_{\cG}$  corresponds to  the  action on $X_{\cG}$ defined in Lemma~\ref{lem-quotient}.

The complication that arises in the study of $\MCG(\fM_{\Phi})$, is that while the intersection of  the image of $\rho_{\cG}$ with  $\Homeo(\fM_{\Phi} ; \whx)$   is trivial,  its   intersection   with $\Homeo(\fM_{\Phi} ; L_{\whx})$ is non-trivial. Thus,  there is not an exact sequence as in \eqref{eq-isoext}.
To characterize the elements in this intersection, introduce the ``discrete'' normalizer subgroup of $\G$, where $\ogamma  \in \fG_{\cG}$ denotes the image of $\gamma$ in $\G_{\cG}$, 
   \begin{equation}\label{eq-Gnormalizer}
\G_{\fN} = \{ \gamma \in \G \mid \ogamma \in \fN_{\cG} \} \ .
\end{equation}

For $\gamma \in \G_{\fN}$ define a left action  on $\wtM_0 \times X_{\cG}$ by  covering translations on the first factor, and the ``adjoint'' action on the second factor:   for $z \in \wtM_0$ and $y \in X_{\cG}$ set
\begin{equation}\label{eq-normaction} 
C(\gamma)(z, y)  =     \left(z \cdot \gamma^{-1} , \Phi(\gamma) \circ \rho_{\cG}( \ogamma^{-1})(y) \right)    \ .
\end{equation}
 We adopt the simplified notation $C(\gamma)(z, y) = (z \cdot \gamma^{-1} , \ogamma \cdot  y \cdot  \ogamma^{-1} )$.
 The  action  $C(\gamma)$  preserves the   equivalence relation  \eqref{eq-coveringaction},  so descends to an  action on the quotient $\G \backslash (\wtM \times X_{\cG})$. This yields a homomorphism denoted $C \colon \fN_{\cG} \to \Homeo(\fM_{\Phi})$.  This construction generalizes  \cite[Example~4.2]{Odden2005}.
 
In terms of the group chain model for $X_{\cG}$, the action $C(\gamma)$ corresponds to conjugation of $\G$ by $\gamma$, which for $\gamma \in \G_{\fN}$  induces an action on the group chain $\cG$,   corresponding to translation of $X_{\cG}$ by $\gamma$  .

 Note that  the action  of $\gamma \in \G_{\fN}$  on the covering of the baseleaf is given by, for $z \in \wtM_0$ and $x $ the basepoint in $X_{\cG}$, 
\begin{equation} \label{eq-normaction2}
C(\gamma)(z, x)   =     \left(z \cdot \gamma^{-1} , \ogamma \cdot  x \cdot  \ogamma^{-1} \right)     
  =     \left(z \cdot \gamma^{-1} , \ogamma   \cdot  \ogamma^{-1}  \cdot x     \right) 
  =     \left(z \cdot \gamma^{-1} , x  \right) \  .
\end{equation}
 Thus, the action $C(\gamma)$ preserves the product  $\wtM_0 \times \{x\}$, and  so the induced action on the quotient 
 $\G \backslash (\wtM \times X_{\cG})$ preserves the baseleaf $L_{\whx}$,  and we have     $C  \colon \G_{\fN} \to \Homeo(\fM_{\Phi} ;  L_{\whx})$.

Also, for   $z \in   \wtM_0$,   $y \in X_{\cG}$ and   $\gamma \in \G_{\fN}$ in the quotient space $\fM_{\Phi} = \G \backslash (\wtM_0 \times X_{\cG})$ we have 
\begin{equation} \label{eq-normaction3}
C(\gamma)[z, y]   =     [z \cdot \gamma^{-1} , \ogamma \cdot  y \cdot  \ogamma^{-1} ]      
  =     [z  ,  y \cdot  \ogamma^{-1} ]
  =     \rho_{\cG}(\ogamma^{-1})[z,y] \  , 
\end{equation}
where $[z,y]$ denotes the equivalence class of $(z,y) \in \wtM_0 \times X_{\cG}$.

 Thus,  deck  translation along $L_{\whx}$ by  $\gamma$ equals translation $ \rho_{\cG}(\ogamma^{-1})$ of $y$ along the fiber $X_{\cG}$, and so   $\rho_{\cG}(\ogamma) \circ C(\gamma)$ is the identity map for $\gamma \in \G_{\fN}$.
 This observation suggests the remedy to the ambiguity in calculating $\MCG(\fM_{\Phi})$, as follows.
 
 Recall that the quotient group $\fA_{\cG} = \fD_{\cG} \backslash \fN_{\cG}$ acts on the right on $X_{\cG}$.
Introduce the composition   
    \begin{equation}
\mu_{\cG} \colon    \fA_{\cG} \times \Homeo(\fM_{\Phi} ; L_{\whx})   \to \Homeo(\fM_{\Phi}) \quad, \quad \mu_{\cG}( \whg , h) = \rho_{\cG}(\whg)^{-1} \circ h \ ,
\end{equation}
   which is continuous, though not a group homomorphism.

Define a left action $\zeta$ of $\G_{\fN}$ on $ \fA_{\cG} \times  \Homeo(\fM_{\Phi} ; L_{\whx}) $ by 
$$\zeta(\gamma) (\whg , h) = ( \rho_{\cG}(\ogamma^{-1})(\whg) ,  C(\gamma) \circ h)  = \left( \whg \cdot \ogamma^{-1} , C(\gamma) \circ h  \right) \  .$$
Define the quotient  space
\begin{equation}\label{eq-qiotientspace}
  \fA_{\cG}  \times_{\zeta}   \Homeo(\fM_{\Phi} ; L_{\whx})  =  \G_{\fN} \backslash \left(     \fA_{\cG} \times \Homeo(\fM_{\Phi} ; L_{\whx})   \right) \ , 
\end{equation}
where $\gamma \in \G_{\fN}$ acts on the first factor by right multiplication by   $\ogamma^{-1}$, and  acts on the second factor by    composition with  $C(\gamma)$.

\begin{prop}\label{prop-nonhomogeneous}
 $\mu_{\cG}$ descends to a continuous injection 
 \begin{equation}\label{eq-tau221}
\omu_{\cG} \colon   \fA_{\cG} \times_{\zeta} \Homeo(\fM_{\Phi} ; L_{\whx})    \longrightarrow \Homeo(\fM_{\Phi})  \ .
\end{equation}
  \end{prop}
 \proof
For $\gamma \in \G_{\fN}$  and  $(\whg, h) \in    \fA_{\cG} \times \Homeo(\fM_{\Phi} ; L_{\whx})$     we have
\begin{eqnarray}\label{eq-equivariant}
\mu_{\cG}(\zeta(\gamma)(\whg , h) )   & = &     \mu_{\cG}( \whg \cdot \ogamma^{-1} ,  C(\gamma) \circ  h )   \\
& = & \rho_{\cG}(\whg \cdot \ogamma^{-1})^{-1} \circ  C(\gamma) \circ  h  \nonumber \\
& = &  \rho_{\cG}(\ogamma \cdot \whg^{-1} ) \circ  C(\gamma) \circ  h \nonumber \\
& = &  \rho_{\cG}(\whg^{-1} ) \circ \rho_{\cG}(\ogamma) \circ  C(\gamma) \circ  h \nonumber  \\
& = &  \rho_{\cG}(\whg^{-1} ) \circ   h =  \mu_{\cG}(\whg, h)  \  .\nonumber
\end{eqnarray}
Thus, the quotient map $\omu_{\cG}$ is well-defined. It is continuous in the quotient topology.

Next consider the kernel of  $\omu_{\cG}$.  Recall that $\fD_{\cG}$ is the kernel of the action   $\rho_{\cG} \colon \fN_{\cG} \to \Homeo(X_{\cG})$, so that $\rho_{\cG} \colon  X_{\cG} \times \fA_{\cG} \to X_{\cG}$ is an effective action.

Suppose that  $h_1, h_2 \in \Homeo(\fM_{\Phi} ; L_{\whx})$ and $\whg_1, \whg_2 \in \fA_{\cG}$ are such that   
 $\omu_{\cG}(\whg_1, h_1) = \omu_{\cG}(\whg_2, h_2)$.
 
Thus 
 $\rho_{\cG}(\whg_1^{-1}) \circ h_1 = \rho_{\cG}(\whg_2^{-1}) \circ h_2$.
      Then for $\whg = \whg_2^{-1} \   \whg_1$, we have 
   $$ h_1 \circ h_2^{-1} = \rho_{\cG}(\whg_1) \circ   \rho_{\cG}(\whg_2^{-1})  = \rho_{\cG}(\whg_2^{-1} \   \whg_1)  = \rho_{\cG}(\whg) \ .$$
     Since $h_1 \circ h_2^{-1} \in    \Homeo(\fM_{\Phi} ; L_{\whx})$, the right action $\rho_{\cG}(\whg)$  must preserve the baseleaf,      
      hence  $ \rho_{\cG}(\whg)  = \rho_{\cG}(\ogamma)$ for some $\gamma \in \G_{\fN}$, and thus $\whg = \ogamma  \in \fA_{\cG}$.   
      Thus we have   $h_1 =  \rho_{\cG}(\ogamma)   \circ h_2$ and so 
\begin{eqnarray}
(\whg_1, h_1)  = (\whg_1 , \rho_{\cG}(\ogamma)   \circ h_2)  & = &  (\whg_1 , C(\gamma^{-1})   \circ h_2)  \\ 
& \sim_{\zeta}  &  ( \rho_{\cG}(\ogamma^{-1})( \whg_1) , h_2) \nonumber \\
& =  &  ( \whg_1 \cdot \ogamma^{-1} , h_2) =  (\whg_1 \cdot   \whg_1^{-1}   \whg_2 , h_2)   =  (\whg_2 , h_2)  \ .\nonumber
\end{eqnarray}
It follows that $\omu_{\cG}$ is injective, as claimed. 
     \endproof

We now restate and complete the proof of Theorem~\ref{thm-nonhomogeneous}.
  
 \begin{thm}\label{thm-normalize}
Let $\fM_{\Phi}$ be a solenoidal manifold. 
Assume the baseleaf $L_{\whx} \subset \fM_{\Phi}$ is simply connected. Then there is an inclusion of topological spaces
\begin{equation}\label{eq-tau21}
\omu_{\cG} \colon   \fA_{\cG} \times_{\zeta}   \MCG(\fM_{\Phi} ; \whx)     \to \MCG(\fM_{\Phi})  \ .
\end{equation}
  \end{thm}
\proof

By Proposition~\ref{prop-nonhomogeneous}, the map 
 $\omu_{\cG} \colon   \fA_{\cG} \times_{\zeta} \Homeo(\fM_{\Phi} ; L_{\whx})    \longrightarrow \Homeo(\fM_{\Phi})$
 is injective. We need to show that the same is true for the induced map on isotopy classes of maps.

We first extend the map $\omu_{\cG}$ to an action of $\fA_{\cG}$ on $\MCG(\fM_{\Phi} ; L_{\whx})$.
That is, me must show that  the action of $\omu_{\cG}$ preserves the isotopy equivalence classes. 
 Let     $h_t \in \Homeo(\fM_{\Phi} ; L_{\whx})$, for $0 \leq t \leq 1$, be an isotopy between $h_0$ and $h_1$, then consider a lift to a family of maps 
  \begin{equation}
\wth_t \colon ( \wtM_0 \times X_{\cG} ; \wtM_0 \times \{x\})  \to (\wtM_0 \times X_{\cG}  ; \wtM_0 \times \{x\})   \quad , \quad 0 \leq t \leq 1  \ . 
\end{equation}

The space $X_{\cG}$ is totally disconnected, so for each $\why \in \fM_{\Phi}$ the trace curve $t \mapsto h_t(\why)$, for $0 \leq t \leq 1$, of the isotopy is a leafwise curve. Thus, for $\whg \in \fA_{\cG}$ the action of $\rho_{\cG}(\whg)$ on $X_{\cG}$ permutes these trace curves, and hence preserves the   isotopy equivalence relation between the maps $h_0$ and $h_1$.
It follows that there is  a well-defined map on the isotopy classes  
\begin{equation}\label{eq-modinduced}
\mu_{\cG} \colon  \fA_{\cG} \times \MCG(\fM_{\Phi} ; L_{\whx})   \to \MCG(\fM_{\Phi})  \ .
\end{equation}
For each $\whg \in \fA_{\cG}$, the map $\rho_{\cG}(\whg)$ is uniformly continuous on $\fM_{\Phi}$,  and therefore acts uniformly continuous on $\Homeo(\fM_{\Phi})$. Thus we obtain a continuous action by  $ \fA_{\cG} $ which descends to an action on the quotient space
\begin{equation}\label{eq-actionmap}
\omu_{\cG} \colon   \fA_{\cG} \times_{\zeta}   \MCG(\fM_{\Phi} ; L_{\whx})     \to  \MCG(\fM_{\Phi})  \ .
\end{equation}
We next show that the  map \eqref{eq-actionmap}  is injective.  
Let  $h, h' \in \Homeo(\fM_{\Phi} ; L_{\whx})$ and $\whg, \whg' \in \fA_{\cG}$  be such that   $\rho_{\cG}(\whg)(h)$ and $\rho_{\cG}(\whg')(h')$ are    isotopic. 
Then for $\whg'' =   \whg' \cdot \whg^{-1}$,  the maps $h$ and $\rho_{\cG}(\whg'') \circ h'$ are isotopic in $\Homeo(\fM_{\Phi})$, and we need to show they are isotopic in $\Homeo(\fM_{\Phi} ; L_{\whx})$.

Let    $H \colon \fM_{\Phi} \times [0,1] \to \fM_{\Phi}$ be an isotopy with  
$$h =  H |_{t=0}  \colon \fM_{\Phi} \times \{0\} \to \fM_{\Phi} \quad {\rm  and} \quad  \rho_{\cG}(\whg'') \circ h' =  H |_{t=1}  \colon \fM_{\Phi} \times \{1\} \to \fM_{\Phi}  \  .$$
The map $h$ preserves the leaf $L_{\whx}$, thus  $h(\whx) \in L_{\whx}$ and hence  the trace $H(\whx, t)$  is a path in $L_{\whx}$. 

Let $h_t(\why) = H(\why, t)$ for $\why \in \fM_{\Phi}$, which defines $h_t \in \Homeo(\fM_{\Phi})$.  We have that $h_t(\whx) = H(\whx, t) \in L_{\whx}$ for $0 \leq t \leq 1$, so $h_t$ is an isotopy in $\Homeo(\fM_{\Phi} ; L_{\whx})$. 
 It follows that 
the map $\omu_{\cG}$ in \eqref{eq-actionmap} is injective.

Finally, by   Proposition~\ref{prop-baseleaf} we have  $\MCG_{BLF}(\fM_{\Phi}) \cong \MCG(\fM_{\Phi} ; \whx)$ so that Theorem~\ref{thm-normalize} is proven. 
\endproof

 Here is the first application of Theorem~\ref{thm-normalize}.
   
      \begin{thm}\label{thm-nonhomogeneous2}
  Let   $\fM_{\Phi}$ be a homogeneous   solenoid  with   simply connected baseleaf $L_{\whx} \subset \fM_{\Phi}$. 
  Then there is a homeomorphism
  \begin{equation}\label{eq-partialproddyn2}
\MCG(\fM_{\Phi}) \cong  \fG_{\cG}  \times_{\zeta}   \MCG(\fM_{\Phi} ; L_{\whx})     \ .
\end{equation}
  \end{thm}
\proof
The discriminant group $\fD_{\cG}$ is trivial exactly when $\fM_{\Phi}$ is a homogeneous space \cite{FO2002,DHL2016a}, so $\fM_{\Phi}$ homogeneous implies that  $\fN_{\cG} = \fA_{\cG} = \fG_{\cG}$.  

Let $h \in \Homeo(\fM_{\Phi})$ and let $\why \in h(L_{\whx}) \cap X_{\cG}$.
The right action $\rho_{\cG}$ of $\fG_{\cG}$ on $\fM_{\Phi}$ is   transitive on $X_{\cG}$ so there exists $\whg \in \fG_{\cG}$ such that $\rho_{\cG}(\whg)(\why) = \whx$. Then $\rho_{\cG}(\whg) \circ h \in \Homeo(\fM_{\cG} ; \whx)$.

It follows that the map $\mu_{\cG}$ in \eqref{eq-tau21} is onto, so the claim follows.
\endproof

  Corollaries~\ref{cor-homogeneous} and \ref{cor-nonhomogeneous}  follow from  Theorems~\ref{thm-summary} and \ref{thm-nonhomogeneous2}.

For a non-homogeneous solenoidal manifold $\fM_{\Phi}$, the group $\fD_{\cG}$ is non-trivial, and the following property of this group implies that its normalizer is never all of $\fG_{\cG}$.
Recall that $\ogamma = \iota_{\cG}(\gamma) \in \fG_{\cG}$ is the image of $\gamma \in \G$ in the completion $\fG_{\cG}$.
 
 \begin{lemma}\cite[Corollary~5.4]{DHL2016a} \label{lem-cnn}
 The normal core $\bigcap_{\gamma \in \G} \ \  \ogamma^{-1} \ \fD_{\cG} \ \ogamma \ = \ \{e\}$ is trivial.
 \end{lemma}
 \proof
Recall that     $\fD_{\cG}$ is identified with a subgroup of $\Homeo(X_{\cG})$, via the left action of $\fG_{\cG}$ on the space, so each non-trivial  $\whh \in \fD_{\cG}$ acts non-trivially on $X_{\cG} \cong \fG_{\cG}/\fD_{\cG}$. 

Thus, for $\whh \ne id$, there is some $\whg \in \fG_{\cG}$ such that  $\whh \cdot \whg \ \fD_{\cG} \ne \whg \ \fD_{\cG}$,  hence  $\whg^{-1} \  \whh \ \whg \not\in \fD_{\cG}$. 
 As the image of $\iota_{\cG} \colon \G \to  \fG_{\cG}$ is dense and $\fD_{\cG}$ is closed, there exists $\gamma \in \G$ such that  $\ogamma^{-1} \ \whh \ \ogamma \not\in \fD_{\cG}$.
\endproof

\begin{cor}
If $\fD_{\cG}$ is non-trivial, then its normalizer $\fN_{\cG}$ is a proper subgroup of $\fG_{\cG}$.
\end{cor}
\proof
For $\gamma \in \G$, set $\fD_{\gamma} = \ogamma^{-1} \ \fD_{\cG} \ \ogamma$. The normalizer of $\fD_{\gamma}$ is the conjugate
  $\fN(\fD_{\gamma})  =  \ogamma^{-1} \ \fN_{\cG}  \ \ogamma$. So if $\fN_{\cG} = \fG_{\cG}$, then every conjugate $\fN(\fD_{\gamma})   = \fG_{\cG}$ and   this contradicts Lemma~\ref{lem-cnn}.
 \endproof
 
We conclude with an observation about the non-homogeneous case.  The intersection of the baseleaf $L_{\whx}$ with the transversal $X_{\cG}$ is identified with the orbit $\cO(x) = \{ \Phi(\gamma)(x) \mid \gamma \in \G\}$  of the $\G$ action on $X_{\cG}$, so  $L_{\whx} \cap X_{\cG} = \cO(x)$. Also, from the proof of Proposition~\ref{prop-baseleaf}, for any $\gamma \in \G$ there exists $h_{\gamma} \in \Homeo_0(\fM_{\cG})$ such that $h_{\gamma}(\whx) = \gamma \cdot x \in L_{\whx} \cap X_{\cG}$.  Thus, for any $\why \in L_{\whx} \cap X_{\cG}$ we have
$$\Homeo(\fM_{\cG} ; \why)  = h_{\gamma} \circ  \Homeo(\fM_{\cG} ; \whx) \circ h_{\gamma}^{-1}   \subset \Homeo(\fM_{\cG}) \ .$$ 
Thus, we can apply Theorem~\ref{thm-normalize} to the conjugate subgroups of $\Homeo(\fM_{\cG} ; \whx)$ to obtain:

\begin{cor}\label{cor-BC}
 Let $\fM_{\Phi}$ be a solenoidal manifold. 
Assume the baseleaf $L_{\whx} \subset \fM_{\Phi}$ is simply connected. Then $\MCG(\fM_{\Phi})$ contains the subgroup generated by all images of $\fN(\fD_{\gamma})  \  \times_{\zeta}  \  \MCG(\fM_{\Phi} ; \why) $ under the corresponding multiplication map $ \mu_{\cG, \why}$ for $\why =\gamma \cdot \whx \in X_{\cG}$.
\end{cor}

That is, 
$\MCG(\fM_{\Phi})$ contains  the embedded image of $\fA_{\cG} \times_{\zeta}   \MCG(\fM_{\Phi} ; \whx)$ and the subgroup generated by all of its conjugates by $\Homeo_0(\fM_{\Phi}, \whx)$.
Note that this subgroup can never be all of $\MCG(\fM_{\Phi})$ in the non-homogeneous case, as otherwise the action of $\Homeo(\fM_{\Phi})$ on $\fM_{\Phi}$ would be transitive, which by  Fokkink and Oversteegen \cite{FO2002} would imply that $\fM_{\Phi}$ is homogeneous space.
  
Corollary~\ref{cor-BC}  should be compared to  Theorem~5.6 by Boyle and  Chuysurichay \cite{BoyleChuysurichay2018}, which describes  the mapping class group of suspensions of subshifts  of finite type.

 \section{Examples}\label{sec-examples}
 
    We give a collection of basic  examples which illustrate the  calculation of $\MCG(\fM_{\Phi})$. The calculation of more difficult or involved examples are graciously left to the reader, or possibly to Claude.
              
      The first step in the calculation of $\MCG(\fM_{\Phi})$ for the suspension of the odometer $(X_{\cG}, \G, \Phi)$ is to determine  $\Comm(\G)$. This group   has fortunately been calculated in many cases  \cite{Studenmund2015,Studenmund2024}. 
      
 The second step  is to calculate $\Comm(X_{\cG}, \G, \Phi, x)$, which consists of local conjugacies of the action $\Phi$ at the basepoint $x \in X_{\cG}$. A  local homeomorphism $h_{loc} \in \Comm(X_{\cG}, \G, \Phi, x)$ determines a commensurator $[\phi] = \Theta_{\cG}(h_{loc}) \in \Comm(\G)$ as in Diagram~\ref{eq-diagram}.

If the action $(X_{\cG}, \G, \Phi)$ satisfies the hypotheses of Theorem~\ref{thm-gottalift} and the base manifold $M$ satisfies the ultra Borel Property, as in      Remark~\ref{rmk-borelex}, then $[\phi] =  \chi([h])$ for some $h \in \Homeo(\fM_{\Phi} ; \whx)$.
 
Finally, we calculate the normalizer $\fN_{\cG}$    to complete the calculation of   $\MCG(\fM_{\Phi})$ when $\fM_{\Phi}$ is  homogeneous, or to obtain a partial calculation of $\MCG(\fM_{\Phi})$ in the non-homogeneous case.

\subsection{Dynamical commensurators}\label{subsec-dyncomm}

We briefly discuss the calculation of  $\Comm(X_{\cG}, \G, \Phi, x)$ for an odometer  $(X_{\cG}, \G, \Phi)$, as this is a key point in the following examples.

Recall from Definition~\ref{def-dyncomm-point} that a dynamical commensurator at $x$ is a   homeomorphism $h_U \colon U \to V$  where  $U, V \subset \fX$ are adapted subsets with $x \in U \cap V$,  $h_U(x) = x$,  and such that $h_U$ induces an isomorphism $\Theta_U \colon \cH_U \to \cH_V$ of their restricted holonomy actions.
  
 For the model space $X_{\cG}$ defined by the group chain $\{\G_{\ell} \mid \ell \geq 0\}$, we take $x = \cap_{\ell > 0} \ U_{\ell}$ and $U= U_{\ell}$ for some $\ell > 0$. Then a commensurator $\phi \in \Comm(\G)$ induces a map $h_{\phi}$, if  for some $\ell_0$,  $\phi$ maps the chain $\{\G_{\ell} \mid \ell \geq \ell_0\}$   to a chain in  $\G$ that is zig-zag equivalent to $\cG$. 
 
 That is,  $\phi$ must satisfy three conditions:
 \begin{enumerate}
\item For   $\ell > \ell_0$ sufficiently large,   there exists $ j > k$ such that $\G_j \subset \phi(\Gamma_{\ell}) \subset \G_{k}$. 
\item  The map $h_{\phi}$ satisfies $h_{\phi}(x) = x$ if  for all $k > 0$ there exists some $\ell > 0$ such that $\phi(\Gamma_{\ell}) \subset \G_{k}$.
\item  The map $h_{\phi}$ is a local homeomorphism if the   inverse commensurator $\phi^{-1}$ maps the chain $\{\G_{\ell} \mid \ell \geq \ell_0'\}$ for some $\ell_0'$,  to a chain in  $\G$ that is zig-zag equivalent to $\cG$.
\end{enumerate}

 Two commensurators $\phi, \phi' \in \Comm(\G)$ satisfy $[h_{\phi}] = [h_{\phi'}] \in    \Comm(X_{\cG}, \G, \Phi, x)$ if for some $\ell_0$ sufficiently large,   the chains 
 $\{\phi(\G_{\ell}) \mid \ell \geq \ell_0\}$ and $\{\phi'(\G_{\ell}) \mid \ell \geq \ell_0\}$ are  zig-zag equivalent.
    
 We use these conditions in the following examples to construct classes in    $\Comm(X_{\cG}, \G, \Phi, x)$.
   
 \subsection{Universal solenoids} \label{subsec-universal}

The calculation of the   group $\MCG(\fM_{\cP} ; \whx)$    by Odden \cite{Odden2005} and Bering and Studenmund \cite{BeringStudenmund2024} are all for   special cases where $\fM_{\cP}$ is the universal solenoidal manifold over an aspherical manifold $M$.  In this case, the map $\Theta_{\cG}$ in Diagram~\ref{eq-diagram} is always defined, and the normalizer $\fN_{\cG} = \fG_{\cG} = \whGamma$ is the full profinite completion of $\G$. We briefly recall this construction.

Recall that $\G$ is finitely generated.  Let  $\cG_f(\G)$ denote the collection of all finite-index subgroups of $\G$,  partially ordered   by inclusion. 
Let $[\G:H]$ denote the index of $H$ in $\G$. Then for each $k > 0$,  the intersection $N_k = \bigcap_{[\G:H] \leq k}   H$ is a normal subgroup which is \emph{characteristic}; that is,   $N_k$ is invariant under every automorphism of $\G$. In particular, each subgroup $N_k$ is normal.

Choose a properly descending subchain   $\{N_{k_\ell} \mid \ell > 0\}$, set $\G_{\ell} = N_{k_{\ell}}$ then form the    normal group chain  $\cG_{u}(\G) = \{ \G = \G_1 \supset \G_2  \supset \cdots \}$.   For    any commensurator $[\phi] \in \Comm(\G)$, it is immediate to check that the three conditions in Section~\ref{subsec-dyncomm} are satisfied, so $\chi$ is onto.  
As a result,  by Theorem~\ref{thm-homogeneous} we obtain a common generalization of the results of Odden in Theorem~\ref{thm-odden}, and of Bering and Studenmund in Theorem~\ref{thm-BS2024}:

\begin{thm}\label{thm-universalthm}
Let $M$ be an ultra Borel manifold, and $\G = \pi_1(M, x)$ be residually finite. Then the map $\chi \colon \MCG(\whM, \whx) \to \Comm(\G)$ is an isomorphism. Moreover, the map $\chi$ induces an isomorphism $\MCG(\whM) \to \whGamma \times_{\zeta} \Comm(\G)$.
\end{thm}

 \subsection{Vietoris solenoids} \label{subsec-vietoris}
 
 A 1-dimensional solenoid has base   $\mS^1$ with $\G = \mZ$. It is a classical result, and an easy exercise, that  $\Comm(\mZ) \cong \mQ^*$ is the group of invertible rational numbers. 
 We   recall some examples of the calculation of $\MCG(\fS_{\Phi})$ where $\fS_{\Phi}$ is the solenoid associated to a group chain $\cG$ in $\mZ$. These examples are   calculated also  by Kwapisz in   \cite{Kwapisz2001,Kwapisz2011};  see also Lind and Ward \cite{LindWard1988}. 

 Let $\cG_u$ be the   group chain defined by   $\G_{m} = (m +1) ! \cdot \mZ$, which is a  universal group chain for $\mZ$.  The resulting solenoid   over $\mS^1$ is denoted  by  $\widehat{\mS^1}$, and the fiber $X_{\cG_u} = \whmZ$ is the profinite completion of $\mZ$.

For every non-zero $a \in \mZ$, let $\mu_a \colon \mZ \to \mZ$ be multiplication by $a$.
Again, it is immediate  to check that the three conditions in Section~\ref{subsec-dyncomm} are satisfied, so   we obtain a homeomorphism $h_a \in \Homeo(\widehat{\mS^1} , \whx)$ which realizes an element of  $\MCG(\widehat{\mS^1} ; \whx ) \cong \mQ^*$. 
The profinite completion $\whmZ$ acts on the right on $\widehat{\mS^1}$, so we obtain $\MCG(\widehat{\mS^1}) \cong \whmZ \times_{\zeta} \mQ^*$.

 Next, suppose we are given a group chain $\cG = \{\G_{\ell} \mid \ell \geq 0\}$ in $\mZ$. For each $\ell$ there is an integer $m_{\ell} > 1$ with $\G_{\ell} = m_{\ell} \cdot \G_{\ell -1}$.
The problem is   to determine when  multiplication by   $\alpha = a/b \in \mQ^*$  satisfies  the three conditions in Section~\ref{subsec-dyncomm}, and so generates an element in $\Comm(X_{\cG}, \mZ, \Phi, x)$.   The answer is given using the notion of  the   {Steinitz}  number associated to   $\vec{m} = \{ m_{\ell} > 1  \mid 1 \leq \ell < \infty\}$, which is the formal infinite product 
     \begin{equation}\label{eq-steinitzorder}
\alpha(\vec{m}) = {\rm lcm} \{ m_1   m_2 \cdots m_{\ell} \mid  \ell > 0\} \ ,
\end{equation}
  where ${\rm lcm}$ denotes the least common multiple of the collection of integers.

A Steinitz number $\alpha$ can be uniquely written   as the formal product over the set of primes   $\Pi$, 
\begin{equation}
\alpha  = \prod_{p \in \Pi} ~ p^{\chi_{\alpha}(p)} \  ,
\end{equation}
where the \emph{characteristic function} $\chi_{\alpha} \colon \Pi \to \{0,1,\ldots, \infty\}$ counts the multiplicity with which a prime $p$ appears in the infinite product $\alpha$.
\begin{defn}\label{defn-type}
Two Steinitz numbers  $\alpha$ and $\alpha'$ are said to be \emph{asymptotically equivalent} if there exist finite integers $m, m' \geq 1$ such that $m \cdot \alpha = m' \cdot \alpha'$, and we then write $\alpha \mor \alpha'$. The equivalence class of $\alpha$ is called its \emph{type}, and denoted by $\tau[\alpha]$. 
\end{defn}

Associate to the collection  $\vec{m} = \{ m_{\ell} > 1  \mid 1 \leq \ell < \infty\}$, covering maps   $p_{\ell} \colon \mS^1 \to \mS^1$ with  degree $m_{\ell}$. Let $\fS_{\vec{m}}$ denote the solenoidal manifold defined by  these maps.   
Then we have:
  \begin{thm}[Bing \cite{Bing1960}, McCord \cite{McCord1965},  Fokkink and Oversteegen \cite{FO2002}] \label{thm-onedimSol}
  Let $\fM$ be an orientable 1-dimensional solenoid. Then there exists $\vec{m}$ such that there is a homeomorphism  $\fM \cong \fS_{\vec{m}}$. Moreover,   solenoids $\fS_{\vec{m}}$ and $\fS_{\vec{m}'}$ are homeomorphic if and only if their types are equal, $\tau[\vec{m}] = \tau[\vec{m}']$.
  \end{thm}

 Thus, for each  choice of a type $\tau(\vec{m})$, the problem is to calculate $\MCG(\fS_{\vec{m}})$. Here are examples.
  
Let $\fS_{\vec{p}}$ be the solenoid associated  an infinite sequence of distinct primes 
  $\vec{p} = \{1 < p_1 < p_2 < \cdots \}$ with $\cG_{\vec{p}}$ the corresponding group chain in $\mZ$.
   That is, we set $\G_{\ell} = p_1 p_2 \cdots p_{\ell} \cdot  \mZ$.  The type of $\cG_{\vec{p}}$ is given by the characteristic function $\chi_{\alpha}$ which takes value $1$ for $p \in \vec{p}$ and 0 otherwise.

  For subgroups $\G_k,\G_\ell \subset \mZ$ we have by the Lagrange theorem that $\G_k \subset \G_\ell $ if and only if the index $|\mZ: \G_\ell|$ divides $|\mZ: \G_k|$. Thus for $\mu_a$ we have that the chain $\cG_{\vec p}$ and its image $\cG_{\vec p}^{\mu_a}$ under $\mu_a$  are zig-zag equivalent if and only if for infinitely many $\ell >0$, there is $i_\ell$ such that $a p_1 \cdots p_{\ell}$ divides $p_1 p_2 \cdots p_{i_\ell}$. Since all primes $p_i$ are distinct, this is not possible, and there is no $a \in \mZ^*$ other than $\pm 1$ which produces a   chain  $\cG_{\vec{p}}^{\mu_a}$ which is zig-zag equivalent to $\cG_{\vec{p}}$.  So,  $\MCG(\fS_{\vec{p}} ; \whx ) \cong \{\pm 1\}$ and so $\MCG(\fS_{\vec{p}}) \cong \mZ_{\vec{p}} \times_{\mZ}    \{\pm1\}$, where 
  \begin{equation} 
    \mZ_{\vec{p}} = \lim_{\longleftarrow}\{\mZ/p_1p_2 \cdots p_{\ell+1} \cdot \mZ \to \mZ/ p_1p_2 \cdots p_\ell \cdot \mZ, \ \ell \geq 0\}
 \end{equation}
is the profinite completion of $\mZ$ with respect to the group chain defined by the sequence $\vec{p}$.

The next example has many generalizations. Choose a prime $p > 1$ and let $\cG_{\whp}$ be the group chain formed by the subgroups $\G_{\ell} = p^{\ell} \cdot \mZ \subset \mZ$ which has type $[p^{\infty}]$. Multiplication by a power of $p$ maps the group chain $\cG_{\whp}$ to a truncated chain $\cG_{\whp}^{\mu_{p^\ell}} = p^\ell \cdot \mZ \supset p^{\ell+1} \cdot \mZ \supset \cdots$. The other two conditions in Section~\ref{subsec-dyncomm} are also easily checked.  Thus,   $\MCG(\fS_{\whp} ; \whx ) \cong \{\pm1\} \times  \mZ$ and these are the only virtual automorphisms which produce a chain zig-zag equivalent to $\cG_{\whp}$. 
It follows that  $\MCG(\fS_{\whp}) \cong \whmZ_p \times_{\mZ}  (\{\pm1\} \times \mZ)$,   
 where $\whmZ_p$ is  the $p$-adic integers.

 The calculation of $\MCG(\fS_{\vec{m}})$ for all other collections of integer sequences $\vec{m}$ gives   results which are analogous to these calculations.  See Kwapisz \cite{Kwapisz2001} for further  calculations.

\subsection{Toral solenoids} \label{subsec-toral}
A \emph{toral solenoid} has  base manifold the $n$-dimensional torus $\mT^n$, and group $\G \cong \mZ^n$. Then $\Comm(\mZ^n) \cong {\rm GL}(\mQ^n)$. 
The manifold  $\mT^n$ is ultra Borel, so  we can apply Theorem~\ref{thm-summary}  to the suspension $\fM_{\Phi}$ of an effective odometer action    $(X_{\cG}, \Phi, \mZ^n)$. 
 
Consider first the universal toral solenoid $\widehat{\mT^n}$ is defined by the group chain $\cG_u$ with $\G_{\ell} = (\ell +1) ! \cdot \mZ^n$. There is an isomorphism of profinite completions $\whG_{\cG_u} \cong \widehat{\mZ^n}$. 
  Every element of    $\Comm(\mZ^n) \cong {\rm GL}(\mQ^n)$   maps the group chain $\cG_u$ to a zig-zag equivalent group chain, so we have   
  \begin{equation}\label{eq-toralauto}
\MCG(\widehat{\mT^n}, \whx) \cong {\rm GL}(\mQ^n) \quad {\rm and} \quad \MCG(\widehat{\mT^n}) \cong \widehat{\mZ^n}  \times_{\zeta}  {\rm GL}(\mQ^n)      \ .
\end{equation}
Note that Odden obtains the result \eqref{eq-toralauto} in \cite[Section~4.5]{Odden2005} for the case $\mT^2$, and also remarks that  $\MCG(\widehat{\mT^2})$ is  isomorphic as groups to the semi-direct product.

 Next,  consider the general case of a toral solenoid $\fM_{\Phi}$, with  group chain    $\cG = \{\G_0 \supset \G_1 \supset \G_2 \cdots \}$ where $\G_0 = \mZ^n$, and  each $\G_{\ell} \subset \mZ^n$ has  rank $n$.  
 For each $\ell > 0$, there exists an invertible $n \times n$ matrix $A_{\ell}$ with integer coefficients and determinant $m_{\ell} = |A_{\ell}| > 1$, such that the map $\phi_{A_{\ell}} \colon \G_{\ell-1} \to \G_{\ell}$ defined by multiplication is an isomorphism.
  Thus, 
$\G_{\ell} = \phi_{A_{\ell}} \circ \cdots \phi_{A_{1}} (\mZ^{n})$.

The $p$-adic toral solenoid $\whmT^n_{p}$ is defined by the group chain $\cG_{(p)}$ with $\G_{\ell} = p^{\ell} \cdot \mZ^n$, for a prime $p \geq 2$. That is, the matrices $A_{\ell}$ above are given by $p$ times the identity,   
and $\whG_{(p)} \cong \whmZ^n_{p}$, where $\whmZ_p$ is the $p$-adic integers. 
   Let  $\phi \in \Comm(\mZ^n) \cong {\rm GL}(\mQ^n)$ and assume that $\phi$ maps   the group chain $\cG_{(p)}$ to   a group chain equivalent to $\cG_{(p)}$. Then $\MCG(\widehat{\mT^n}, \whx) \cong {\rm GL}(\mZ^n [\frac{1}{p}])$, and so 
     $\MCG(\whmT^n_{p}) \cong \whmZ^n_{p}  \times_{\mZ^n} {\rm GL}(\mZ^n [\frac{1}{p}])$.

The next simplest  case is where each $A_{\ell}$   is a diagonal matrix  for all $\ell > 0$. For example, given  a   collection of disjoint primes $S = \{p_1, p_2, \ldots , p_n\}$, let $\cG_{S}$ be the group chain defined by the matrices $A_{\ell}$ whose diagonal entries are the $\ell$-th powers of primes in $S$ in the same order, and are 0 off the diagonal. The associated toral solenoid $\whmT^n_{S}$ is the product of 1-dimensional solenoids,  
$\ds \whmT^n_{S} \cong \fS_{\widehat{p_1}} \times \cdots \times \fS_{\widehat{p_n}}$. A map which is a return equivalence of the action of $\mZ^n$  on $X_{\cG_S}$    must preserve the factors for distinct primes, so
\begin{equation}\label{eq-prodprimes}
\Comm(X_{\cG_S}, \mZ^n, \Phi, x) \cong {\rm GL}(\mZ [\frac{1}{p_1}]) \times \cdots \times {\rm GL}(\mZ [\frac{1}{p_n}]) \cong  \mQ^* \times \cdots \times \mQ^* \ .
\end{equation}
 
One can modify this construction by choosing a multiplicity $n_i$ for each prime, then taking the product of solenoids 
$\whmT^n_S = \whmT^{n_1}_{p_1} \times \cdots \times \whmT^{n_k}_{p_k}$ where $n = n_1 + \cdots + n_k$. We then obtain
\begin{equation}
\Comm(X_{\cG_S}, \mZ^n, \Phi, x) \cong {\rm GL}(\mZ^{n_1} [\frac{1}{p_1}]) \times \cdots \times {\rm GL}(\mZ^{n_k} [\frac{1}{p_k}]) \ , 
\end{equation}
and so we have
 \begin{equation}\label{eq-fancytoral}
 \MCG(X_{\cG_S}) \cong 
\left\{\widehat{\mZ_{p_1}}^{n_1} \times \cdots \times \widehat{\mZ_{p_k}}^{n_k} \right\}   \times_{\zeta} \left\{ {\rm GL}(\mZ^{n_1} [\frac{1}{p_1}]) \times \cdots \times {\rm GL}(\mZ^{n_k} [\frac{1}{p_k}]) \right\} \ .
\end{equation}

The next class  of toral solenoid examples to consider are  when $A_{\ell} = A$ for all $\ell > 0$, and all eigenvalues of the integer $n \times n$ matrix $A$ are larger than 1, and $n \geq 2$. 
 Then $A$ induces a proper self-covering map $\phi_A \colon \mT^n \to \mT^n$. Let $\mT^n_A$ denote  the inverse limit solenoid.  The map $\phi_A$ induces a self-map $\oPhi_A \in \Homeo(\mT^n_A ; \whx)$, so that $\MCG(\mT^n_A ; \whx)$ always contains a factor of $\mZ$.  
 The surprising fact is that the group  $\MCG(\mT^n_A ; \whx)$ can be much larger and complex than this obvious subgroup of expanding maps. 
 
For example, when $n=2$,   Cabezas and  Petite give in   \cite[Theorem 3.3]{CabezasPetite2024}   a complete description of  $\MCG(\mT^2_A ; \whx)$. In addition, \cite[Example 3.4]{CabezasPetite2024}   describes one of the cases for $n > 2$.
The calculation depends on the algebraic number theory properties of the roots of the characteristic polynomial $p_A(\lambda)$ of $A$.  See also the work of Ha   and Lee  \cite{HaLee2023}.

When $n > 2$, in most cases the  complete calculation of  $\MCG(\mT^n_A ; \whx)$ is unknown, even in the case where 
$p_A(\lambda)$ is separable with distinct roots.  
The calculation of $\Comm(X_{\cG_A}, \mZ^n, \Phi, x)$ is equivalent to calculating the endomorphisms of the subgroup   $ G_{A^t}  =  \{ (A^t)^k\vec{x} \mid \vec{x} \in \mZ^n \ , \ k \in \mZ  \} \subset \mQ^n$  determined by the adjoint matrix $A^t$.  The most complete results  to date have been obtained by  Sabitova in \cite{Sabitova2024b}.

 \subsection{Nilpotent solenoids}\label{subsec-nilpotent}
 
A   \emph{nilmanifold} is a closed connected manifold $M$ whose fundamental group $\G = \pi_1(M,x)$ is   nilpotent, and   the universal cover $\wtM$ is contractible.  A nilmanifold is ultra Borel and its fundamental group is torsion free.

For a group chain $\cG$ in $\G$ which determines a presentation $\cP$,   the associated \emph{nilpotent solenoid} is denoted $\whM_{\cP}$.  
By Corollary~\ref{cor-nilp} and Theorem~\ref{thm-surjective}, we can   calculate $\MCG(\whM_{\cP}, \whx)$ in terms of its dynamical commensurator group  $\Comm(X_{\cG}, \G, \Phi, x)$.

A finitely generated nilpotent group $\G$ is residually finite, so  the universal   solenoid $\whM$  over $M$ contains a simply connected leaf. Then proceeding as in the abelian case, we have:
\begin{thm}\label{thm-nilmcg}
Let $M$ be a nilmanifold with $\G = \pi_1(M,x)$. Then
$\MCG(\widehat{M}) \cong \whGamma \times_{\zeta} \Comm(\G)$.
\end{thm}
 
The calculation of $\Comm(\G)$   for a finitely generated nilpotent group $\G$ reduces to the calculation of the automorphisms of the rational Lie algebra ${\mathfrak g}_{\G}$ associated to the Mal'cev completion of $\G$. This is explained, for  example, in  \cite[Chapter 4]{Baumslag1971}. Thus, the calculation is  analogous to the abelian case discussed above, although the calculation of $\Aut({\mathfrak g}_{\G})$  can be much more subtle.

Another class of nilpotent solenoid examples uses the fact   that a  finitely generated torsion-free nilpotent group   admits an embedding  $\sigma \colon \G \to {\rm SL}(\mZ^m)$ for some $m > 0$ sufficiently large, as shown  by Jennings \cite{Jennings1955} (see also \cite{Formanek1970}). Fix such an embedding $\sigma$, and choose a prime $p > 0$. Then define, for each $\ell \geq 0$, the normal subgroup of finite index
\begin{equation}\label{eq-nilpotentfiltration}
\G_{\sigma,p,\ell} = \ker \left\{ {\rm mod}_p \circ \sigma \colon \G \to {\rm SL}(\mZ^m) \rightarrow {\rm SL}(\mZ^m/p^{\ell} \mZ^m) \right\} \ .
\end{equation}
This defines a normal group chain $\cG_{\sigma,p}$ with $\cap_{\ell > 0} \G_{\sigma, p, \ell} = \{e\}$. Let $\whGamma_{\cG_{\sigma,p}}$ denote the profinite group associated to the inverse limit of the quotients $\G/\G_{\G_{\sigma,p,\ell}}$.  The associated tower of coverings $\cP_{\sigma,p}$ of $M$ defines the solenoidal manifold  $\whM_{\sigma,p}$.
Then as before, we have:
\begin{thm}\label{thm-pnilmcg}
Let $M$ be a nilmanifold with $\G = \pi_1(M,x)$, $\sigma \colon \G \to  {\rm SL}(\mZ^n)$ an embedding, and $p > 1$ a prime. Then
$\MCG(\whM_{\sigma,p}) \cong \whGamma_{\cG_{\sigma,p}} \times_{\zeta} \Comm(X_{\cG_{\sigma,p}}, \G, \Phi, x)$. 
\end{thm}
 The calculation of $\Comm(X_{\cG_{\sigma,p}}, \G, \Phi, x)$ is related to the representation theory associated with  the action of $\G$ on the $p$-localization $\whGamma_{\cG_{\sigma,p}}$ of the profinite completion, as studied  for example  in     \cite{Moore1965}.
 Some examples are calculated in \cite{HL2026a}.

 Recall that a group $\G$ is said to be (finitely) \emph{non-coHopfian} \cite{HLvL2021} if there exists a proper, finite-index  self-embedding $\vp \colon \G \to \G$. Such a map $\vp$ is called a \emph{renormalization} of $\G$, and we say that $\G$ is renormalizable. For example, homogeneous nilpotent groups are renormalizable \cite{Cornulier2016,Dere2017}, though there are many examples of renormalizable groups which are not nilpotent, as discussed in \cite{HLvL2021}.

Suppose that $\G$ admits a proper self-embedding $\vp \colon \G \to \G$. For $\ell \geq 0$, let $\G_{\ell} = \vp^{\ell}(\G) = \vp \circ \cdots \circ \vp(\G)$ be the image of the $\ell$-th  iteration of $\vp$. Also assume that $\cap_{\ell > 0} \G_{\ell} = \{e\}$.
Then the group chain  $\cG_{\vp} = \{\G_0 \supset \G_1 \supset \G_2 \supset \cdots\}$ is a \emph{strong scale} in the sense of \cite{NekrashevychPete2011}. 

Let $(X_{\vp}, \G, \Phi)$ denote the odometer associated to the chain $\cG_{\vp}$.
If the groups $\G_{\ell}$ are normal in $\G$, then van Limbeek showed in \cite{WvL2021} that $\G$ must be abelian. Thus, if $\G$ is non-abelian, then  the subgroups $\G_{\ell}$ are never normal, the associated discriminant group $\fD_{\cG}$ is not trivial, and so the inverse limit space $X_{\vp}$ is not a pro-finite group. 

Let $\whM_{\vp}$ denote the nilpotent solenoid associated to the coverings defined by the group chain $\cG_{\vp}$.  Then,   $\whM_{\vp}$ has a simply connected leaf, since $\ker \cG_\phi = \{e\}$. 
 By Theorem \ref{thm-surjective} 
\begin{equation}
\MCG(\whM_{\vp} ; \whx) \cong \Comm(X_{\vp}, \G, \Phi, x) \ .
\end{equation}
The map $\vp$ induces a self map $\widetilde{\vp} \colon M \to M$, which in turn defines $\whvp \in \Homeo(\whM_{\vp} ; \whx)$ so that the group $\Comm(X_{\vp}, \G, \Phi, x)$ always contains at least a copy of $\mZ$. However, as with the abelian case analyzed by Sabitova \cite{Sabitova2024b},   there are many more ``exotic'' classes in $\MCG(\whM_{\vp})$ which arise from number theoretic properties of the map $\vp$, as discussed in   \cite{HL2026a}. 
 
 The calculation of $\MCG(\whM_{\vp})$ from $\MCG(\whM_{\vp} ; \whx)$ using  Corollary~\ref{cor-nonhomogeneous} depends on understanding the normalizer $\fN_{\cG}$, which is not a straightforward problem in general.

 \subsection{Hyperbolic solenoids}\label{subsec-hyperbolic}
 Let $\Sigma_g$ denote the closed orientable surface of genus $g \geq 2$,  let  $x \in \Sigma_g$ be a basepoint, and then set $\G_g = \pi_1(\Sigma_g, x)$. 
Theorem~\ref{thm-odden}   by Odden  together with Theorem \ref{thm-universalthm}    states that  the characteristic map $\chi$ yields  an isomorphism $\MCG(\widehat{\Sigma_g})  \cong \whG_g \times_{\zeta} \Comm(\G_g)$, where   $\widehat{\Sigma_g}$ is the universal hyperbolic solenoid and $\whG_g$ is the profinite completion of $\G_g$. 
 This result can be generalized in two directions. The following follows from  Theorem~\ref{thm-surjective}.
 \begin{thm}
Let $\cG$ be a group chain in $\G_g$ such that $\ker(\cG) = \{e\}$.  Then $\MCG(\widehat{\Sigma_g} ; \whx) \cong \Comm(X_{\cG}, \G_g, \Phi, x)$.
\end{thm}
  
 The second generalization extends the results of Bering and Studenmund in    \cite{BeringStudenmund2024}.
 Let $M$ be a closed connected $n$-manifold with a metric of negative sectional curvatures, for $n \geq 2$. Then $M$ is ultra  Borel. Let $x \in M$ and $\G = \pi_1(M, x)$.
  If the metric on $M$ has constant curvature, then $\G$ is linear; that is, $\G$ admits an embedding into a connected Lie group, and so $\G$ is residually finite.  For more general hyperbolic groups,   Tholozan and Tsouvalas  have constructed in \cite{TT2025} examples which are Gromov hyperbolic and not residually finite.  Thus, we assume that $\G$ is linear.

Let $\whGamma$ be the profinite completion of $\G$, then the universal odometer action of $\G$ on $\whGamma$ is effective as $\G$ is residually finite. We then obtain the corresponding result to Theorem 5.6 in   \cite{BeringStudenmund2024}:

 \begin{thm}\label{thm-negcomm}
Let $M$ be a closed connected manifold with constant negative curvature, and let $\G = \pi_1(M,x)$. Then
$\MCG(\widehat{M}) \cong \whGamma \times_{\zeta} \Comm(\G)$. 
\end{thm}
\proof
For the universal odometer action, we have $\Comm(X_{\cG}, \G, \Phi, x) \cong \Comm(\G)$, and the map $\chi$ is then an isomorphism by Remark~\ref{rmk-borelex}.
\endproof


\end{document}